\documentclass[10pt]{amsart}

\usepackage{amsmath, amscd, amssymb}
\usepackage[frame,cmtip,arrow,matrix,line,graph,curve]{xy}
\usepackage{graphpap, color,bbm, verbatim}
\usepackage[mathscr]{eucal}
\usepackage[normalem]{ulem}
\usepackage{graphicx}
\usepackage{tikz}
\usetikzlibrary{arrows}

\numberwithin{equation}{section}

\newcommand{\CC}{\mathbb{C}}

\newcommand{\PP}{\mathbb{P}}
\newcommand{\QQ}{\mathbb{Q}}

\newcommand{\ZZ}{\mathbb{Z}}

\newcommand{\bC}{\mathbf{C}}

\newcommand{\cal}{\mathcal}

\def\cA{{\cal A}}

\def\cC{{\cal C}}

\def\cE{{\cal E}}
\def\cF{{\cal F}}
\def\cH{{\cal H}}

\def\cL{{\cal L}}
\def\cM{{\cal M}}

\def\cO{{\cal O}}

\def\cU{{\cal U}}

\def\fM{\mathfrak{M}}

\def\fX{\mathfrak{X}}

\def\fQ{\mathfrak{Q}}
\def\fP{\mathfrak{P}}

\newcommand{\tL}{\widetilde{L}}

\def\mapright#1{\,\smash{\mathop{\lra}\limits^{#1}}\,}

\def\virt{^{\mathrm{vir}}}

\def\lra{\longrightarrow}

\def\and{\quad{\rm and}\quad}

\def\and{\quad\text{and}\quad}
\def\mapright#1{\,\smash{\mathop{\lra}\limits^{#1}}\,}

\def\PP{\mathbb{P}}
\def\CC{\mathbb{C}}

\def\lra{\longrightarrow}
\def\mapright#1{\,\smash{\mathop{\lra}\limits^{#1}}\,}
\def\cO{\mathcal{O}}

\def\Spec{\mathrm{Spec}}

\def\loc{_{\mathrm{loc}} }
\def\vilo{\virt\loc}

\input xy
\xyoption{all}

\DeclareMathOperator{\coker}{coker} 
 
\DeclareMathOperator{\Ext}{Ext} 
\DeclareMathOperator{\Hom}{Hom}

\newtheorem{prop}{Proposition}[section]
\newtheorem{theo}[prop]{Theorem}
\newtheorem{lemm}[prop]{Lemma}
\newtheorem{coro}[prop]{Corollary}

\theoremstyle{definition}
\newtheorem{exam}[prop]{Example}
\newtheorem{defi}[prop]{Definition}

\newtheorem{rema}[prop]{Remark}

\title[LG/CY correspondence via quasi-maps]{A new method toward the Landau-Ginzburg/Calabi-Yau correspondence via quasi-maps}

\author{Jinwon Choi}
\address{Department of Mathematics and Research Institute of Natural Sciences, Sookmyung Women's University, Seoul 04310, Korea}
\email{jwchoi@sookmyung.ac.kr}

\author{Young-Hoon Kiem}
\address{Department of Mathematics and Research Institute of Mathematics, Seoul National University, Seoul 08826, Korea}
\email{kiem@math.snu.ac.kr}

\thanks{Jinwon Choi was supported by NRF grants NRF-2015R1C1A1A01054185 and NRF-2018R1C1B6005600. Young-Hoon Kiem was partially supported by Samsung Science and Technology Foundation grant SSTF-BA1601-01.
}

\def\mgn{\fM_{g,m}}
\def\mgnst{\overline{\cM}_{g,m}}

\def\fQf{\mathfrak{Q}_+ }
\def\hqmapf{{\mathfrak{X}}_+ }
\def\DM{Deligne-Mumford }
\def\fQ{\mathfrak{Q} }

\newcommand{\olog}{\omega^\mathrm{log}}
\def\otw{\omega^{\mathrm{tw}}}
\def\tw{\mathrm{tw}}
\def\ps{\mathrm{ps}}
\def\modmop{\fQ^{\epsilon=0^+}_+ }
\def\beq{\begin{equation} }
\def\eeq{\end{equation} }

\def\loc{_{\mathrm{loc}} }
\def\fX{\mathfrak{X} }

\begin{document}
\begin{abstract}
The Landau-Ginzburg/Calabi-Yau correspondence claims that the Gromov-Witten invariant of the quintic Calabi-Yau 3-fold should be related to the Fan-Jarvis-Ruan-Witten invariant of the associated Landau-Ginzburg model via wall crossings. In this paper, we consider the stack of quasi-maps with a cosection and introduce sequences of stability conditions which enable us to interpolate between the moduli stack for Gromov-Witten invariants and the moduli stack for Fan-Jarvis-Ruan-Witten invariants.
\end{abstract}
\maketitle
\section{Introduction}\label{sec1}
\subsection{Landau-Ginzburg/Calabi-Yau correspondence}
In 1993, Witten in \cite{Wit93} introduced Landau-Ginzburg (LG for short) models as diagrams of morphisms
\[\xymatrix{
\CC^N\ar[rr]^W\ar[dr] && \CC\\
&[\CC^N/G]\ar[ur] }\]
where $G$ is a finite group and $W$ is a $G$-invariant polynomial. In this paper, we will concentrate on the case
$$W:\CC^5\to \CC\text{ defined by }W(x_1,\dots,x_5)=\sum_{i=1}^5 x_i^5 \and G=\ZZ_5\le \CC^*.$$

There is a corresponding Calabi-Yau (CY for short) model for an LG model.
Consider the stack
\[ [\CC^5\times \CC/\CC^* ]\]
where $\CC^*$ acts with weights $(1,\dots,1,-5)$ and let
\[ \hat{W}(x_1,\dots, x_5,p)=p(\sum_{i=1}^5 x_i^5). \]
be the invariant polynomial on $\CC^5\times \CC$.
Then we have two open substacks
\[
\cO_{\PP^4}(-5)=(\CC^5-0)\times \CC/\CC^* \subset  [\CC^5\times \CC/\CC^* ] \supset \CC^5\times (\CC-0)/\CC^* =\CC^5/\ZZ_5.
\]
On the left side where $x=(x_1,\dots,x_5)\ne 0$, called the CY model, the open substack is $\cO_{\PP^4}(-5)$ and $\hat{W}$ descends to a section $W=\sum_{i=1}^5x_i^5$ of $\cO_{\PP^4}(5)$. Let $$Y_+=W^{-1}(0)\subset \PP^4$$ denote the quintic Calabi-Yau 3-fold defined by $W=0$.
The invariant we like to calculate on this CY side is the Gromov-Witten (GW for short) invariant $GW(Y_+)$ which virtually enumerates stable maps $C\to Y_+$ from prestable curves.

On the right side where $p\ne 0$, called the LG model, the open substack is $\CC^5/\ZZ_5$ with $\ZZ_5\le \CC^*$ (as fifth roots of unity). We can let $p=1$ and thus $\hat{W}$ becomes $W(x_1,\dots,x_5)=\sum_{i=1}^5 x_i^5$ to give us the LG model
\[ Y_-:=([\CC^5/\ZZ_5]\mapright{W} \CC).\]
The invariant for the LG side is the Fan-Jarvis-Ruan-Witten (FJRW for short) invariant which is defined as the degree of Witten's top Chern class
\[ FJRW(Y_-)=\#\{(C,L)\,|\, C\text{ twisted stable curve}, L^5\cong \otw_C\}\cap e(R^1\pi_*\cL^{\oplus 5})\]
if $\pi_*\cL=0$ where $\cL$ is a universal line bundle \cite{FJR}. Here, $L$ is an invertible sheaf on $C$, $\otw_C:=\omega_C(\sum q_i)$ is the sheaf of sections of the dualizing sheaf $\omega_C$ possibly with poles of order 1 at the \emph{orbifold marked points} $q_i$, and $\pi$ is the natural projection from the universal curve to the moduli stack.
In general, the FJRW invariant can be defined algebro-geometrically by the cosection localization \cite{CLL,KiemLi}.

The Landau-Ginzburg/Calabi-Yau correspondence is a conjectural equivalence between the GW theory of $Y_+$ and the FJRW theory of $Y_-$. Chiodo and Ruan in \cite{ChiodoRuan} proved that in genus zero the two theories are related by an explicit symplectic transformation followed by analytic continuation. The result is extended to complete intersections in projective space in \cite{Clader}.

Classical mirror symmetry compares the GW invariant of $Y_+$ with the variation of Hodge structures around the large complex structure limit point $\lambda=\infty$ for the family $(\sum_{i=1}^5x_i^5-\lambda\prod_{i=1}^5x_i=0)$ with $\lambda\in \PP^1$
while the FJRW invariant for the LG model $Y_-$ is related to that at the Gepner point $\lambda=0$.
By the analytic continuation of the variations of Hodge structures on $\PP^1-\{0,1,\infty\}$, one may expect that the GW invariants of $Y_+$ should be related to the FJRW invariant of $Y_-$. 
So, the following question seems quite natural.

\medskip
{\it Can we relate $GW(Y_+)$ with $FJRW(Y_-)$ by wall crossings?
}
\medskip

On the CY side, a sequence of stability conditions, called the $\epsilon$-stabilities with $\epsilon>0$, was introduced by Toda \cite{Toda} interpolating the moduli of stable maps and stable quotients (See \cite{MOP}). Ciocan-Fontanine and Kim in \cite{CFKwcf1, CFKwcf2} worked out the wall crossing of the invariants as $\epsilon$ varies from $\infty$ to $0^+$, where $\epsilon=\infty$ means $\epsilon$ being sufficiently large so that the $\epsilon$-stable quasi-maps gives the Gromov-Witten theory and $\epsilon=0^+$ means $\epsilon$ being positive and sufficiently close to zero.
On the LG side, a corresponding theory for $\epsilon<0$ and wall crossing of invariants as $\epsilon$ varies from $-\infty$ to $0^-$ was worked out by Fan, Jarvis and Ruan in \cite{FJR15} and Ross and Ruan in \cite{RossRuan} for $g=0$, where $\epsilon=-\infty$ means $\epsilon$ being sufficiently small so that the $\epsilon$-stable quasi-maps gives the FJRW theory and $\epsilon=0^-$ means $\epsilon$ being negative and sufficiently close to zero. However, it has been unclear how the $\epsilon=0^+$-stability on the CY side and the $\epsilon=0^-$-stability on the LG side are related from the A-model point of view.

As we will see below in Section \ref{sec1.7} (See \cite{ChangLi, CLL}), all these moduli stacks are open substacks of the stack of quadruples $(C,L,x,p)$ where $C$ is a prestable or twisted semistable curve, $L$ is a line bundle on $C$, $x\in H^0(C,L^{\oplus 5})$ and $p\in \Hom_C(L^5, \otw_C)$.
Note that $\otw_C$ (resp. $\olog_C$) allows simple poles only for \emph{orbifold (resp. smooth)} marked points and hence $\otw_C=\omega_C$ when there are no orbifold marked points. When all the marked points are orbifold points, $\otw_C$ is $\olog_C$ (resp. $\omega_{\mathrm{log}}$) in the notation of \cite{ChangLi} (resp. \cite{FJR15,RossRuan}).
All the invariants can be defined as integrals on the cosection-localized virtual cycles \cite{KiemLi}.

\medskip

The purpose of this paper is to show that there are further stability conditions interpolating the $\epsilon=0^+$-stability on the CY side and the $\epsilon=0^-$-stability on the LG side. We call the new stability the $\delta$-stabilities with $\delta>0$ (resp. $\delta<0$) for the CY side (resp. LG side). In this paper, we will show that
\begin{enumerate}
\item for given topological type $(g,m,d)$, there are only finitely many $\delta$-walls where the moduli space of $\delta$-stable quasi-maps changes;
\item for general $\delta>0$ (resp. $\delta<0$), the stack of $\delta$-stable quadruples $(C,L,x,p)$ with $x\ne 0$ (resp. $p\ne 0$) is a separated \DM stack equipped with a cosection-localized virtual cycle whose support is proper;
\item the $\delta=\pm\infty$-stability is very close to $\epsilon=0^\pm$-stability. The precise relationship $\delta=\infty$-stability and $\epsilon=0^+$-stability for genus zero will be clarified in Section \ref{sec:MOPinf};
\item a quadruple  $(C,L,x,p)$ is $\delta=0^+$-stable (resp. $\delta=0^-$-stable) if and only if $\bar L$ is (Gieseker-)stable over $\bar C$ with respect to the ample line bundle $\olog_{\bar C}$
where $\rho:C\to \bar C$ is the stabilization of $C$ (resp. the stabilization of the coarse moduli space of the twisted semistable curve $C$) and $\bar L=\rho_*L$ (resp. $\bar L=\rho_* (L^{-5}\otw_C)$); there are no constraints on $x\ne 0$ and $p$ (resp. $x$ and $p\ne 0$);
\item via the torus localization (See \cite{CLK}) for the cosection-localized virtual cycles, the $\delta=0^\pm$-invariant is given by the residue of an integral on the moduli stack $\bar P$ of pairs $(\bar C,\bar L)$ of stable curves $\bar C$ and stable sheaves $\bar L$ on $\bar C$ (See \cite{Simpson}), where the formulae on $\delta=0^-$ and $\delta=0^+$ are of the same form.
\end{enumerate}

To prove the full LG/CY correspondence, we need a wall crossing formula for cosection-localized virtual cycles, which gives a relation between the virtual invariants as $\delta$ varies. But in the present paper, we do not provide a wall crossing formula. We give a description of the moduli spaces and leave pursuing the derivation of a wall crossing formula as future work.

In Section \ref{sec:WCg0}, we show that the $\epsilon$- and $\delta$-wall crossings are all given by regular morphisms when $g=0$. If we further specialize to the case of $d=1$, we show that these morphisms are in fact blowups. As a byproduct, we obtain a new construction of Fulton-MacPherson configuration space of $\PP^1$ as a sequence of blowups from a projective bundle over the moduli space $\overline{\cM}_{0,m}$ of stable curves via the moduli spaces of $\epsilon$- and $\delta$-stable quadruples (Example \ref{exam:FM}).

\medskip

Perhaps the moral of this paper may be phrased as follows: Through magnifying glasses, we discovered that between $\epsilon=0^+$ and $\epsilon=0^-$, there is another line of wall crossings which we call the $\delta$-line, where the $\epsilon=0^\pm$-stabilities are very close to the $\delta=\pm\infty$-stabilities and the invariants for the $\delta=0^+$ and $\delta=0^-$-stabilities are given by the same residue formula (See Figure \ref{fig1}).

\begin{figure}[htb]
\setlength\intextsep{0pt}
\begin{center}
\includegraphics{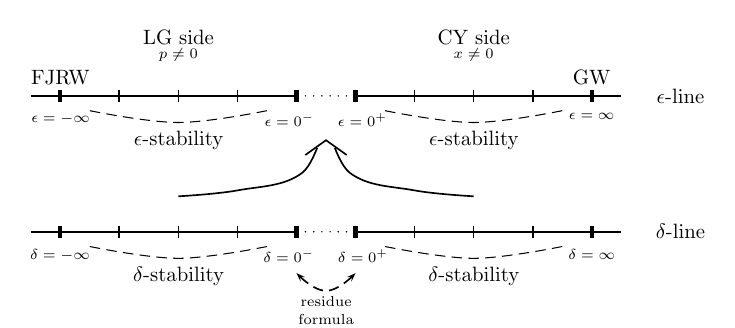}
\end{center}
\caption{$\epsilon$- and $\delta$-line of stability conditions}\label{fig1}
\end{figure}

\medskip
\subsection{GW and FJRW invariants by cosection localization}\label{s0.2}

Recall that the GW invariant enumerates stable maps $f:C\to Y_+\subset \PP^4$,
which amounts to giving a line bundle $L=f^*\cO_{\PP^4}(1)$ and a multi-section $x=(x_1,\dots,x_5)\in H^0(C,L)^{\oplus 5}$ such that $x:\cO_C^{\oplus 5}\to L$ is surjective. The map $f$ is a stable map to $Y_+$ if and only if $\sum_{i=1}^5x_i^5=0$ and $\olog_C\otimes L^3$ is ample.
On the other hand, the FJRW invariant enumerates twisted stable curves $C\in \mgn^{\tw}$ with spin structures (or $p$-fields) $p:L^5\mapright{\cong} \otw_C$ by Witten's top Chern class. 
When there are no orbifold marked points, the isomorphism $p$ can be thought of as a section in $H^0(L^{-5}\omega_C)$.

In \cite{ChangLi}, Chang and Li studied the GW theory of $Y_+$ by enlarging the moduli stack by including sections $p\in H^0(L^{-5}\omega_C)$ which are called \emph{$p$-fields}. They considered the stack $\fX_+^{\epsilon=\infty}$ of quadruples $(C,L,x, p)$ defined by
\[ \fX_+^{\epsilon=\infty}=\{
(C,L,x,p)\,|\, (C,L,x)\in \overline{M}_{g,m}(\PP^4,d) ,~ p\in H^0(L^{-5}\omega_C)
\}.\]
Here $\overline{M}_{g,m}(\PP^4,d)$ denotes the moduli stack of stable maps to $\PP^4$, i.e. $x:\cO_C^{\oplus 5}\to L$ is surjective and $\olog_C\otimes L^3$ is ample.
They proved that the moduli stack $\fX_+^{\epsilon=\infty}$ is equipped with a perfect obstruction theory whose obstruction sheaf $Ob_{\fX_+^{\epsilon=\infty}}$ admits a cosection
\begin{equation}\label{defcos} \sigma=\bar{\partial}\hat W:Ob_{\fX_+^{\epsilon=\infty}}\lra \cO_{\fX_+^{\epsilon=\infty}},\quad \bar{\partial}\hat W|_{(C,L,x,p)}(\dot{x},\dot{p})=\dot{p}\sum_{i=1}^5x_i^5+p\sum_{i=1}^55x_i^4\dot{x}_i,\end{equation}
where $\dot{x}=(\dot{x}_i)\in H^1(L^{\oplus 5})$ and $\dot{p}\in H^1(L^{-5}\omega_C)$.
The degeneracy locus $D(\sigma)=\sigma^{-1}(0)$ consists of stable maps to $Y_+$ so that $D(\sigma)$ is proper. The cosection localization of \cite{KiemLi} then defines a virtual cycle $[\fX_+^{\epsilon=\infty}]\vilo$ supported in $D(\sigma)$, which is proved to be equivalent to $[ \overline{M}_{g,m}(Y_+,d)]\virt$ up to sign. Hence $[\fX_+^{\epsilon=\infty}]\vilo$ gives the GW invariant of the Fermat quintic $Y_+$.

There is a similar story on the LG side. 
Let $\fX_-^{\epsilon=-\infty}$ be the moduli stack parametrizing quadruples $(C,L,x,p)$, where $C\in\mgn^{\tw}$ is a twisted stable curve (See \cite{AV}), $L$ is a line bundle with fixed nontrivial orbifold structures on orbifold marked points, $x\in H^0(L^{\oplus 5})$, and an isomorphism $p: L^5\cong \otw_C$. (See \cite{FJR,CLL} or Section \ref{sec:LG} for more detail.) The perfect obstruction theory and the cosection on $\fX_-^{\epsilon=-\infty}$ are defined similarly as on  $\fX_+^{\epsilon=\infty}$.
The degeneracy locus $D(\sigma)=\sigma^{-1}(0)$ in  $\fX_-^{\epsilon=-\infty}$  is precisely the locus where $x=0$, which is finite over $\mgn^{\tw}$.
Chang, Li and Li in \cite{CLL} proved that integrals on $[\fX_-^{\epsilon=-\infty}]\vilo$ gives the FJRW invariant of $Y_-$ up to sign when all marked points are orbifold points.

Note that both the GW and FJRW invariants are defined by moduli stacks of quadruples $(C,L,x,p)$ with $C$ either prestable or twisted stable, $L\in \mathrm{Pic}(C)$, $x\in H^0(C,L^{\oplus 5})$ and $p\in \Hom_C(L^{5},\otw_C)$, satisfying certain stability conditions. Let $\fX$ denote the stack of all quadruples $(C,L,x,p)$ (without any stability conditions) and let $\fX_+$ (resp. $\fX_-$) denote the open substack defined by $x\ne 0$ (resp. $p\ne 0$).
By the cone stack construction (See \cite{ChangLi}), $\fX$ (resp. $\fX_\pm$) is an algebraic stack. Observe that $\fX_\pm^{\epsilon=\pm\infty}$ are open substacks of $\fX_\pm$. We will see below that there are many stability conditions on $\fX_\pm$ that give rise to open \DM substacks equipped with localized virtual cycles. The moduli stacks for the CY side are all contained in $\fX_+$ with $C\in \mgn^{\ps}$ prestable while  those for the LG side lie in $\fX_-$ with $C\in \mgn^{\tw}$ twisted semistable.

\subsection{Comparison of GW and FJRW via wall crossing}\label{s0.3}

In this subsection, we review the known sequence of stability conditions on $\fX_\pm$ and introduce a new sequence, called the $\delta$-stabilities, on $\fX_\pm$. Both of $\epsilon$ and $\delta$ takes values from $\infty$ to $0^+$ for $\fX_+$ and from $-\infty$ to $0^-$ for $\fX_-$.

\subsubsection{$\epsilon$-wall crossing}\label{s0.3.1}

For $\epsilon>0$, a quadruple $(C,L,x,p)\in \fX_+$ for $C\in \mgn^{\ps}$, $L\in \mathrm{Pic}^d(C)$, $x\in H^0(L^{\oplus 5})=\Hom_C(\cO_C^{\oplus 5},L)$ and $p\in H^0(L^{-5}\omega_C)=\Hom_C(L^5,\omega_C)$ is called \emph{$\epsilon$-stable} if
\begin{enumerate}
\item $\olog_C\otimes L^\epsilon$ is ample;
\item the cokernel of $x:\cO^{\oplus 5}_C\to L$ has 0-dimensional support disjoint from special points (i.e. marked points or nodes);
\item $\epsilon\cdot\mathrm{length}_z(\mathrm{coker}(x))<1$ for all $z\in C$.
\end{enumerate}
Note that we do not require any condition on the $p$-field $p$.

Let $\fX_+^{\epsilon}$ denote the open substack in $\fX_+$ of $\epsilon$-stable quadruples. The definition of $\epsilon$-stability (without $p$) is due to Toda in \cite{Toda} generalizing the MOP stability in \cite{MOP}.
When $\epsilon$ is sufficiently large ($\epsilon=\infty$), no base points are allowed and an $\epsilon$-stable quadruple is nothing but a stable map to $\PP^4$ with $p$-field so that we get back to $\fX_+^{\epsilon=\infty}$ for the GW invariant in Section \ref{s0.2} as we should. Furthermore, the moduli stacks $\fX_+^\epsilon$ vary only at finitely many values of $\epsilon>0$, called the {\em walls}, if we fix the topological type $(g,m,d)$. For a non-wall $\epsilon>0$, $\fX_+^\epsilon$ are all equipped with cosection-localized virtual cycles $[\fX_+^\epsilon]\vilo$ whose supports are proper.

When $\epsilon=0^+$, the $\epsilon$-stability equals the MOP stability which reads as \begin{enumerate}
\item $\olog_C\otimes L^a$ is ample for all $a>0$;
\item the cokernel of $x:\cO^{\oplus 5}_C\to L$ has 0-dimensional support disjoint from special points.
\end{enumerate}
As we will see below, if we drop the phrase ``disjoint from special points'' and allow exceptional components (i.e. rational components $E$ with $\olog_C|_E\cong \cO_E$) in the support, we get the $\delta=\infty$-stability that we will introduce.

On the LG side, for $\epsilon<0$, a quadruple $(C,L,x,p)\in \fX_-$ for $C\in \mgn^{\tw}$, $L\in \mathrm{Pic}^d(C)$, $x\in H^0(L^{\oplus 5})$ and $p\in H^0(C,L^{-5}\otw_C)=\Hom_C(L^5,\otw_C)$ is called \emph{$\epsilon$-stable} if
\begin{enumerate}
\item $\otw_C\otimes \tilde{L}^{|\epsilon|}$ is ample where $\tilde{L}=L^{-5}\otw_C$;
\item the cokernel of $p:L^{5}\to \otw_C$ has 0-dimensional support disjoint from special points;
\item $|\epsilon|\cdot\mathrm{length}_z(\mathrm{coker}(p))<1$ for all $z\in C$.
\end{enumerate}
Note that we do not require any condition on the $x$-field $x$. Obviously, the condition (1) may be rephrased as the ampleness of  $(\otw_C)^{1-\epsilon}\otimes L^{5\epsilon}$.

Let $\fX_-^{\epsilon}$ denote the open substack in $\fX_-$ of $\epsilon$-stable quadruples. This definition of $\epsilon$-stability is due to Fan, Jarvis and Ruan in \cite{FJR15} and Ross and Ruan in \cite{RossRuan}. Ross and Ruan in \cite{RossRuan} further studied the wall crossing for $g=0$. When $\epsilon$ is sufficiently negative ($\epsilon=-\infty$), no base points for $p$ are allowed and an $\epsilon$-stable quadruple is nothing but a spin structure with $x$-field so that we get back to the moduli stack $\fX_-^{\epsilon=-\infty}$ for the FJRW invariant in Section \ref{s0.2}. As before, the moduli stacks $\fX_-^\epsilon$ vary only at finitely many values of $\epsilon<0$, called walls, upon fixing the topological type $(g,m,d)$. For a non-wall $\epsilon<0$, $\fX_-^\epsilon$ are all equipped with cosection-localized virtual cycles $[\fX_-^\epsilon]\vilo$ whose supports are proper.

When $\epsilon=0^-$, the $\epsilon$-stability reads as \begin{enumerate}
\item $\otw_C\otimes \tilde L^{a}$ is ample for all $a>0$ where $\tilde{L}=L^{-5}\otw_C$;
\item the cokernel of $p:L^5\to \otw_C$ has 0-dimensional support disjoint from special points.
\end{enumerate}
If we only drop the phrase ``disjoint from special points,'' we get the $\delta=-\infty$-stability that we will introduce in Section \ref{sec:LG}. Again obviously, the condition (1) may be rephrased as the ampleness of $\otw_C\otimes L^a$ for all $a<0$.

\subsubsection{$\delta$-stability}\label{s0.3.2}
What we propose in this paper is that there are sequences of wall crossings between $\epsilon=0^+$ and $\epsilon=0^-$ in both sides, which we call the $\delta$-wall crossing.

For $\delta>0$, we define the $\delta$-stabilities on $\fX_+$ as follows (See Definition \ref{def-delstab2} and Lemma \ref{lem-use}):
A quadruple $(C,L,x,p)\in \fX_+$ with $C\in \mgn^{\ps}$, $L\in \mathrm{Pic}^d(C)$, $x\ne 0\in H^0(L^{\oplus 5})$ and $p\in \Hom_C(L^5,\omega_C)$  is called \emph{$\delta$-stable} if
\begin{enumerate}
\item $\olog_C\otimes L^a$ is ample for all $a >0$;
\item $(\bar L,\bar x)$ is a $\delta$-stable pair on $\bar C$ in Le Potier's sense (See \cite{LePotier}) with respect to the ample line bundle $\omega_{\bar C}^{\mathrm{log}}$, where $\rho:C\to \bar C$ is the stabilization morphism, $\bar L=\rho_*L$ and $\bar x=\rho_*x$. See Definition \ref{def-delstab1} for the notion of $\delta$-stable pairs.
\end{enumerate}

Let $\fX_+^{\delta}$ denote the open substack of $\fX_+$ which consists of $\delta$-stable quadruples.
We prove in Section \ref{sec3} that when $d+\delta\ge g-1$, $\fX_+^\delta$ is a separated \DM stack over $\CC$ for any $\delta>0$ except for finitely many values (once we fix $g,m,d$), called the walls (See Theorem \ref{thm2-1}).
By \cite{ChangLi, CFK, KiemLi}, for a non-wall $\delta>0$, $\fX_+^\delta$ are all equipped with perfect obstruction theories and cosections $\sigma:Ob_{\fX_+^\delta}\to \cO_{\fX_+^\delta}$ of the obstruction sheaves as well as the cosection-localized virtual fundamental cycles $[\fX_+^\delta]\virt_{\mathrm{loc}}$.

We also show that if $d+\delta\ge 3(g-1)+m$, the degeneracy locus $D(\sigma)$ of the cosection $\sigma$ in $\fX_+^\delta$ is proper so that we obtain a sequence of invariants of the quintic 3-fold by integrating cohomology classes against the localized virtual cycle $[\fX_+^\delta]\virt_{\mathrm{loc}}$ (See Theorem \ref{thm:properdeg}).

When $\delta$ is sufficiently large, denoted by $\delta=\infty$, the $\delta$-stability reads as
\begin{enumerate}
\item $\olog_C\otimes L^a$ is ample for all $a>0$;
\item $C$ is quasi-stable and the degree of $L$ on a rational bridge is 1;
\item the cokernel of $\bar x:\cO_{\bar C}^{\oplus 5}\to \bar L$ has 0-dimensional support.
\end{enumerate}
Obviously this is very close to the $\epsilon=0^+$-stability above. The wall crossing from the $\epsilon=0^+$-stable triples $(C,L,x)$ to the $\delta=\infty$-stable triples (without cosection localization) is described in Section \ref{sec:MOPinf}.

For the GW invariants of $Y_+$, we may assume that $2g-2+m$ is coprime to $d-g+1$ where $d=\deg L$ because we can add a marked point and cancel its effect by the formula
\begin{equation}\label{edila} (2g-2+m)\cdot [\overline{M}_{g,m}(Y_+,d)]\virt = [\overline{M}_{g,m+1}(Y_+,d)]\virt\cap \psi_{m+1}. \end{equation}

Under the assumption that $2g-2+m$ is coprime to $d-g+1$, there is no strictly $\delta=0^+$-semistable quadruples and the $\delta=0^+$-stability of $(C,L,x,p)$ is equivalent to the stability of $(\bar C,\bar L)$. We have a proper \DM stack $\bar P=\bar{P}_{g,m,d}$ of pairs $(\bar C,\bar L)$ where $\bar C\in \mgnst$ is a stable curve and $\bar L$ is a (Gieseker-)stable sheaf on $\bar C$ with respect to the ample line bundle $\omega_{\bar C}^{\mathrm{log}}$ by the general construction of Simpson (See \cite{Simpson}). Moreover, by a result of Faltings \cite[Theorem 4.1]{Faltings}, $\bar P$ is smooth.
We will see (See the proof of Theorem \ref{thm:properdeg}) that when $(C,L,x,p)$ is $\delta=0^+$-stable and $d\ge 3(g-1)+m$, we have the vanishing $p=0$. Hence $\fX_+^{\delta=0^+}$ is in fact the stack of triples $(C,L,x)$ with $x\ne 0$ where $(\bar C,\bar L)\in \bar P$. It is easy to see that $(\bar C, \bar L)$ determines $(C,L)$ uniquely by inserting a rational bridge at a node where $\bar L$ is not locally free and hence $\fX_+^{\delta=0}$ is the projectivization of a cone stack over $\bar P$. Applying the torus localization formula in \cite{CLK}, we find that (See Proposition \ref{p6.2n})
\begin{equation}\label{eK1} [\fX_+^{\delta=0^+}]\vilo = \mathrm{res}_{t=0}\frac{[\bar P]}{e(R\pi_*(\cL^{\oplus 5}\oplus  \cH om(\cL^{5},\omega_{\cC/\bar P})))}\end{equation}
where $\pi:\cC\to \bar P$ denotes the universal curve and $\cL$ is the universal sheaf on $\cC$.
Here $e(\cdot)$ stands for the equivariant Euler class of the perfect complex.

There is analogous $\delta$-stability on the LG side (Section \ref{sec:LG}). We consider the stack $\fX_-$ of quadruples $(C,L,x,p)$ of twisted semistable curves $C\in \mgn^{\tw}$, line bundles $L$ on $C$, $x$-fields $x\in H^0(L^{\oplus 5})$ and nonzero $p$-fields $p\in \Hom_C(L^5, \otw_C)$.
Exactly in the same manner as $\fX_+$ in the CY side discussed above, we introduce the $\delta$-stabilities on $\fX_-$ and discuss the wall crossing for $\delta<0$. The perfect obstruction theories, cosections, localized virtual cycles and so on can all be constructed in the same way.
When $\delta=-\infty$, the stability is exactly the same as the $\epsilon=0^-$-stability without the phrase ``disjoint from special points.'' When $\delta=0^-$ and $d<-\frac15 (g-1+m)$, we will find that $x=0$. We may assume the numerical condition such that there are no strictly $\delta=0^-$-semistable quasi-maps as there is also dilaton equation for FJRW invariants \cite[Theorem 4.29]{FJR}. Then we have a forgetful map from $\fX_-^{\delta=0^-}$ to $\bar{P}$ and $[\fX_-^{\delta=0^-}]\vilo$ is given by
\begin{equation}\label{IntroeK2}
 [\fX_-^{\delta=0^-}]^{\rm vir}_{\rm loc} = r\cdot\mathrm{res}_{t=0}\frac{[\bar{P}_{g,m,\tilde d}]}{e(R\pi_*(\cL^{\oplus 5}\oplus  \cH om(\cL^{5},\omega_{\cC/\bar{P}})))},\end{equation}
where $r$ is the degree of the finite morphism sending $L$ to $\tilde L = L^{-5}\otw$. Note the symmetry in formulae \eqref{eK1} and \eqref{IntroeK2}.

The key point is that when $\delta=0^\pm$, the moduli stack $\fX_+^{\delta=0^+}$ and $\fX_-^{\delta=0^-}$ are projective bundle over a \DM stack of pairs of a stable curve and a line bundle on the curve, because $p=0$ for $\delta=0^+$ and $x=0$ for $\delta=0^-$. Therefore we can apply the torus localization formula.

It is our hope to shed insightful light on the LG/CY correspondence through the $\epsilon$ and $\delta$ stabilities.
In summary we will have the following diagram:
\[\xymatrix{
\fX_+^{\epsilon=\infty}\ar@{<.>}[d]_{\epsilon\text{-wall crossing}} & \fX_-^{\epsilon=-\infty}\ar@{<.>}[d]^{\epsilon\text{-wall crossing}} \\
\fX_+^{\epsilon=0^+}\ar@{<.>}[d]  & \fX^{\epsilon=0^-}_-\ar@{<.>}[d]  \\
\fX_+^{\delta=\infty}\ar@{<.>}[d]_{\delta\text{-wall crossing}}  & \fX_-^{\delta=-\infty}\ar@{<.>}[d]^{\delta\text{-wall crossing}}  \\
\fX_+^{\delta=0^+}\ar@{.>}[d]_{\text{forget }x} & \fX_-^{\delta=0^-}\ar@{.>}[d]^{\text{forget }p}\\
\bar{P} & \bar{P}
}\]
The cosection-localized virtual cycle for the top left gives the GW invariant while that for the top right gives the FJRW invariant. The wall crossing formulas for both sides will relate the GW invariant and the FJRW invariants with integrals on $\bar P=\bar{P}_{g,m,d}$. By combining them, one may deduce a correspondence between the two invariants.

In summary, as the $\epsilon$-wall crossings on both sides have been studied by many authors, the LG/CY correspondence may be achieved by working out the wall crossings for $\delta$-stable quasi-maps on each side. In this paper, we focus on developing the moduli theory for $\delta$-stability. To complete the full LG/CY correspondence, we will need to derive a wall crossing formula for the cosection-localized virtual cycles. We hope to address the issue of wall crossing formulas in future work.

\medskip

As an application, we also study how the moduli spaces change on the CY side in the special case of $g=0$ (prior to cosection localization). By the stability conditions it is easy to check that $p=0$ on the CY side when $g=0$. So the moduli space $\fX_+^{\epsilon/\delta}$ is the same as the moduli space of triples $(C,L,x)$ which is denoted by $\fQ_+^{\epsilon/\delta}$. Then we show that there are contraction morphisms (Theorem \ref{thm:g0})
\begin{equation}\label{eq:cont}
\fQ_{+}^{\epsilon=\infty} \lra \cdots \lra \modmop\lra  \fQ_{+}^{\delta=\infty} \lra \cdots \lra  \fQ_{+}^{\delta=0^+}.
\end{equation}
and that $\fQ_{+}^{\delta=0^+}$ is a projective bundle over the moduli space $\overline{\cM}_{0,m}$ provided that $d+1$ and $m-2$ are coprime.

When we further specialize to the case where $d=1$, these contraction morphisms are all given by blowups (Theorem \ref{thm:g0d1}). As a byproduct, we obtain the Poincar\'{e} polynomials for $\overline{M}_{0,m}(\PP^{n-1},1)$. When the target space is $\PP^1$, the blowup maps \eqref{eq:cont} give a new construction of the Fulton-MacPherson configuration space $\PP^1[m]$ from $\overline{\cM}_{0,m}$ (Example \ref{exam:FM}).

We remark that there is a master space approach for the LG/CY correspondence for all genera by Chang, Li, Li and Liu \cite{CLLL}. Chiodo and Ruan in \cite{ChiodoRuan} proved the LG/CY correspondence in genus zero. Also in genus zero, Lee, Priddis and Shoemaker \cite{LPS} establish a proof of LG/CY correspondence assuming the crepant transformation conjecture. In \cite{GuoRoss}, Guo and Ross verified the Landau-Ginzburg/Calabi-Yau correspondence in genus one. The correspondence for higher genera remains open.

\subsection{Outline of the paper} This paper is organized as follows. In Section \ref{sec1.5} and Section \ref{sec1.7}, we review the general theory of quasi-maps and define a cosection on the stack of quasi-maps with $p$-fields following \cite{ChangLi}. In Section \ref{sec2}, we review the theory of $\delta$-stable pairs by Le Potier. In Section \ref{sec3}, we construct the moduli stack of $\delta$-stable quasi-maps and the cosection-localized virtual cycle. In Section \ref{sec:LG}, we describe the parallel theory on the LG side. In Section \ref{sec:WCg0}, we study the wall crossing when genus is zero. In Section \ref{sec:residue}, we show that when $\delta=0^\pm$, the invariants on both sides are given by the same residue formula. Finally in Section \ref{sec:insertion}, we give the definition of the descendent invariants for cosection-localized class of the moduli space of quasi-maps.

\medskip

\noindent\textbf{Acknowledgement}: We thank Huai-Liang Chang, Emily Clader, Tyler Jarvis, Bumsig Kim, Jun Li and Yongbin Ruan for useful discussions.

\medskip

\section{Direct image cone and the stack of quasi-maps}\label{sec1.5}
In this section, we define the stack $\fQ_+=\fQ_+(g,m,d)$ of quasi-maps of degree $d>0$ to a projective space $\PP^{4}$ for $g, m\in \ZZ_{\ge 0}$ and recall the result of Ciocan-Fontanine and Kim \cite{CFK}, independently Chang and Li \cite{ChangLi} that it is an algebraic stack and that any open \DM substack of $\fQ_+$ is equipped with a perfect obstruction theory.

\begin{defi}\label{def1}
An $m$-pointed \emph{prestable} curve is a connected reduced curve $C$ which has at worst nodal singularities, together with $m$ distinct smooth marked points $q_1,\dots, q_m$ in $C$. An $m$-pointed prestable curve $(C,q_1,\dots, q_m)$ is \emph{stable} (resp. \emph{semistable}) if $\olog_C:=\omega_C(\sum_{i=1}^m q_i)$ is ample (resp. nef).
\end{defi}

When it is obvious, we will write $C$ instead of $(C,q_1,\dots,q_m)$ for an $m$-pointed prestable curve.

\begin{defi}
A \emph{quasi-map} to $\PP^{4}$ is a tuple $(C, L, x)$ where $C$ is an $m$-pointed prestable curve $C$, $L$ is an invertible sheaf on $C$ and $x:\cO_C^{\oplus 5}\to L$ is a nonzero homomorphism. The sum of degrees $\deg L|_{C_i}$ where $C=\cup C_i$ is the irreducible decomposition is called the \emph{degree} of the quasi-map $(C,L,x)$.
\end{defi}

\begin{rema}\label{rmkn}
Although we have defined quasi-maps only when the target is $\PP^4$, we can also define it for general $\PP^{n-1}$ by setting $x$ to be in $H^0(L)^{\oplus n}$. All results in this paper which do not involve a cosection hold for general $n$.
\end{rema}

A stable map to $\PP^{4}$ is a quasi-map.
\begin{lemm}\label{lem116.1} Let $(C, L,x)$ be a quasi-map to $\PP^{4}$. Suppose $x$ is surjective. Then the morphism $\phi_{x}:C\to \PP^{4}$ induced by $x$ is a stable map if and only if $\olog_C\otimes L^3$ is ample.
\end{lemm}
The proof of this lemma is an elementary exercise.

Let us define the stack of all quasi-maps to $\PP^{4}$ of degree $d$ over $m$-pointed prestable curves of genus $g$. First of all, we consider the stack $\mathfrak{M}_{g,m}^{\ps}$ of $m$-pointed prestable curves of genus $g$ whose sections over a scheme $S$ are flat proper morphisms $\cC\to S$ of finite type with $m$ disjoint smooth sections $q_1,\dots,q_m$ whose geometric fibers are $m$-pointed prestable curves of genus $g$. It is well known that $\mathfrak{M}_{g,m}^{\ps}$ is a smooth algebraic stack.

Next the stack $\fP_+:=\mathfrak{P}_+(g,m,d)$ of line bundles of degree $d$ over prestable curves is defined as a stack over the category of schemes over $\CC$ whose sections over a scheme $S$ are pairs of $(\cC\to S)\in\mathfrak{M}_{g,m}^{\ps}$ and $\cL\in \mathrm{Pic}(\cC)$ of relative degree $d$ with obvious pull-backs as arrows.
It is also well known that $\mathfrak{P}_+$ is also a smooth algebraic stack and there is a forgetful morphism
\[ \mathfrak{P}_+\lra \mathfrak{M}_{g,m}^{\ps}.\]

The stack $\fQ_+=\fQ_+(g,m,d)$ of quasi-maps is now defined as the stack whose sections over a scheme $S$ are triples $(\cC\to S, \cL, x_S)$ where $(\cC\to S,\cL)\in \mathfrak{P}_+(S)$ and $x_S:\cO^{\oplus 5}_{\cC}\to \cL$ is a homomorphism which is not trivial on any fiber of $\cC\to S$. Given $(\pi:\cC\to S,\cL)\in \mathfrak{P}_+(S)$, it was shown in \cite[Proposition 2.2]{ChangLi} that the groupoid of all sections $x$ is represented by the cone
\begin{equation}\label{econ}
C(\pi_*\cL^{\oplus 5}):=\mathrm{Spec}_S\left(\mathrm{Sym} R^1\pi_*[(\cL^\vee)^{\oplus 5}\otimes \omega_{\cC/S}]\right)
\end{equation}
where $\mathrm{Sym}\cF$ denotes the symmetric algebra of a coherent sheaf $\cF$.
This construction defines  a cone stack over $ \mathfrak{P}_+$ and $\fQ_+$ is obtained by deleting the vertex of the cone.
Therefore $\fQ_+$ is an algebraic stack and we have the forgetful morphism $\fQ_+\to \mathfrak{P}_+$.

\begin{theo}[{\cite{CFK,ChangLi}}]\label{theo-sec1.5.1}
The stack $\fQ_+$ is an algebraic stack. Any open \DM substack of $\fQ_+$ admits a perfect obstruction theory induced from \eqref{eq-pot}.
\end{theo}
\begin{proof}
The proof is the same as the proof of \cite[Proposition 3.1]{ChangLi} (See also \cite[Section 5]{CFK}).
Let $\cL_{\fQ_+}$ denote the universal line bundle over the universal curve $\pi:\cC_{\fQ_+}\to\fQ_+$.
By applying \cite[Proposition 2.5]{ChangLi} to the evaluation morphism $\mathfrak{e}:\cC_{\fQ_+}\to\cL^{\oplus 5}_{\mathfrak{Q}_+} $ induced from the universal section $\mathfrak{u}\in \Gamma(\cC_{\fQ_+},\cL^{\oplus 5}_{\mathfrak{Q}_+})$, we have a perfect relative obstruction theory \begin{equation}\label{eq-pot}\phi:\mathbb{L}^\vee_{\fQ_+/\mathfrak{P}_+}\to R\pi_*\cL^{\oplus 5}_{\mathfrak{Q}_+}\end{equation}
for $\fQ_+ \to \mathfrak{P}_+$.
Since $\mathfrak{P}_+$ is a smooth algebraic stack, this gives us a perfect obstruction theory for $\fQ_+$.
\end{proof}

\begin{coro}\label{cor-sec1.5.2}
Let $\mathfrak{U}$ be an open substack of $\fQ_+$ such that \begin{enumerate}
\item $\mathfrak{U}$ is a \DM stack;
\item $\mathfrak{U}$ is separated;
\item $\mathfrak{U}$ is proper.
\end{enumerate}
Then $\mathfrak{U}$ is equipped with a perfect obstruction theory and hence the virtual fundamental class $[\mathfrak{U}]\virt$ of dimension
\[ (3g-3+m) + g+ [5(d-g+1)-1]= -(g-1)+m+ 5d. \]
\end{coro}

One can check (2) and (3) using the valuative criterion (See Section \ref{sec:pf}). For (1), if we can write $\mathfrak{U}$ as the quotient $V/G$ of a scheme $V$ by a reductive group $G$, it suffices to show that the stabilizer groups are all finite and reduced by \cite[Corollary 2.2]{Edidin}.

\begin{exam} {(Stable maps)}
By Lemma \ref{lem116.1}, the moduli stack $\overline{M}_{g,m}(\PP^{4},d)$ of stable maps is an open substack of $\fQ_+$ which is a separated proper \DM stack \cite[Proposition 2.7]{ChangLi}. It was shown in \cite[Section 5.3]{CFK} and \cite[Lemma 2.8]{ChangLi} that the perfect obstruction theory of Behrend-Fantechi coincides with that obtained in Theorem \ref{theo-sec1.5.1}.
\end{exam}
\begin{exam}\label{ex-stquo} {(Stable quotients)}
If a stable quotient $\cO_C^{\oplus 5}\to Q$ has rank $4$ so that the kernel $S$ is invertible, the dual $\cO^{\oplus 5}_C\to S^{\vee}$ of the inclusion $S\hookrightarrow \cO^{\oplus 5}_C$ defines an object in $\fQ_+$. These objects form an open substack $\overline{Q}_{g,m}(\PP^{4},d)$ of $\fQ_+$ which is a proper separated \DM stack by Marian, Oprea and Pandharipande \cite{MOP}. See also \cite{Toda}.
\end{exam}
In Section \ref{sec3}, we will introduce the notion of $\delta$-stable quasi-maps and find that the substack $\fQ^\delta_+$ of $\delta$-stable quasi-maps for general $\delta$ is also an open substack which is a proper separated \DM stack.

\def\pgmd{{\mathfrak{P}_{g,m,d}} }


\section{GSW model for $\fQ_+$}\label{sec1.7}
In \cite{ChangLi}, Chang and Li further enlarged the moduli stack $\mathfrak{Q}_+(g,0,d)$ to include an additional section, called the \emph{$p$-field}. The cosection localization technique of \cite{KiemLi} then enables us to define localized invariants which are related to curve counting invariants on a quintic Calabi-Yau 3-fold $X$. The same construction also applies when there are marked points.

\begin{defi}
A \emph{$p$-field} of a quasi-map $(C,L,x)\in \fQf$ is a section $$p\in \Hom_C(L^5,\omega_C)= H^0(C,L^{-5}\otimes \omega_C).$$
The stack $\fX_+=\fX_+(g,m,d)$ of quasi-maps with $p$-fields is defined as the stack whose sections over a scheme $S$ are quadruples $(\cC\to S,\cL,x_S,p_S)$ where $(\cC\to S,\cL,x_S)\in \fQf(S)$ and $p_S\in H^0(\cC,\cL_S^{\otimes -5}\otimes \omega_{\cC/S})$. Arrows are defined by pull-backs.\end{defi}

By the direct image cone construction of \cite[Section 2]{ChangLi} again, we have the following.
\begin{prop}
$\fX_+$ is an algebraic stack equipped with a relative perfect obstruction theory
\[  \mathbb{L}^\vee_{\hqmapf/\mathfrak{P}_+}\lra R\pi_*(\cL_{\hqmapf}^{\oplus 5}\oplus [\cL^{ -5}_{\hqmapf}\otimes \omega_{\pi}])      \]
of $\fX_+\to \mathfrak{P}_+$ where $\pi:\cC_{\hqmapf}\to \hqmapf$ is the universal curve and $\cL_{\hqmapf}$ is the universal line bundle over $\cC_{\hqmapf}$.\end{prop}

The natural map
\[
\hat{W}:\cL^{\oplus 5}_{\hqmapf}\oplus (\cL^{-5}_{\hqmapf}\otimes \omega_{\pi}) \lra \omega_{\pi}
, \quad (x_i,p)\mapsto p\sum_ix_i^5\]
has derivative
\begin{equation}\label{eq-hpot} \sigma=\bar{\partial}\hat{W}:Ob_{\hqmapf/\mathfrak{P}_+}=R^1\pi_*\left(\cL_{\hqmapf}^{\oplus 5}\oplus [\cL^{ -5}_{\hqmapf}\otimes \omega_{\pi}]\right) \lra R^1\pi_*\omega_{\pi}\cong \cO_{\hqmapf}
\end{equation}
defined by $$(\dot{x}_i,\dot{p})\mapsto \dot{p}\sum x_i^5+p\sum 5x_i^4\dot{x}_i, \quad \text{for } (\dot{x}_i,\dot{p})\in H^1(C,L)^{\oplus 5}\oplus H^1(C,L^{-5}\otimes\omega_C)$$ at $(x_i,p)\in H^0(C,L)^{\oplus 5}\oplus H^0(C,L^{ -5}\otimes\omega_C)$. Moreover, Chang and Li show in \cite[Section 3.4]{ChangLi} that the map \eqref{eq-hpot} can be lifted to a cosection $\sigma:Ob_{\hqmapf}\lra \cO_{\hqmapf}$.

The locus where the cosection $\sigma$ in \eqref{eq-hpot} is not surjective is called the \emph{degeneracy locus} of $\sigma$ and denoted by $D(\sigma)$.

\begin{lemm}\label{lem-0118}
Suppose $\widehat{\mathfrak{U}}$ is an open substack of $\hqmapf$ which is a \DM stack. If the support of the multi-section $x:\cO_C^{\oplus 5}\to L$ contains the support of $p$ for all $(C, L, x,p)\in \widehat{\mathfrak{U}}$, then the degeneracy locus $D(\sigma)$ of $\sigma$ in $\widehat{\mathfrak{U}}$ is the closed substack of $\widehat{\mathfrak{U}}$ which consists precisely of $(C,L,x,p)\in \widehat{\mathfrak{U}}$ satisfying $p=0$,  $\sum_ix_i^5=0$ where $x=(x_1,\dots,x_5)$.
\end{lemm}
The proof is identical to that of Proposition 3.4 in \cite{ChangLi}. When $D(\sigma)$ is proper, we can apply the localized virtual cycle construction of \cite{KiemLi}.
\begin{coro}\label{cor-0118}
If $D(\sigma)$ in $\widehat{\mathfrak{U}}$ is proper, we have the localized virtual fundamental class $$[\widehat{\mathfrak{U}}]\virt_{\mathrm{loc}}\in A_m(D(\sigma)),$$
integrals against which define invariants for $Y_+$.
When $m=0$, the degree of $[\widehat{\mathfrak{U}}]\virt_{\mathrm{loc}}$ defines an \emph{invariant of quasi-maps with $p$-fields}.
\end{coro}
For instance, if $\mathfrak{U}$ is an open substack of $\fQf$ which is a proper separated \DM stack and $\widehat{\mathfrak{U}}=\mathfrak{U}\times_{\fQf}\hqmapf$ satisfies the assumption of Lemma \ref{lem-0118}, then $D(\sigma)$ is a closed substack of $\mathfrak{U}$ since $p=0$ and hence $D(\sigma)$ is proper.

In \cite[Section 5]{ChangLi}, it was proved that when $\widehat{\mathfrak{U}}=\overline{{M}}_{g,0}(\PP^4,d)^p$ is the moduli stack of stable maps with $p$-fields, the localized invariant
\[ \deg [\overline{{M}}_{g,0}(\PP^4,d)^p]\virt_{\mathrm{loc}}\]
coincides with the ordinary GW invariant of a quintic 3-fold up to sign $(-1)^{5d-g+1}$.

\begin{rema}\label{rem:cfk}
The theory of quasi-maps to GIT quotients is developed in \cite{CKM}. When the target is $\PP^{n-1}$ (See Remark \ref{rmkn}), the quasi-map in \cite{CKM} is precisely the stable quotient in Example \ref{ex-stquo}. Upon restricting the target, one obtains the virtual cycle associated the stack of quasi-maps to a quintic 3-fold. We expect this also coincides with the cosection-localized virtual cycle \[[\overline{Q}_{g,m}(\PP^{n-1},d)^p]\virt_{\mathrm{loc}}\] up to sign. When the genus is zero, since there are no nonzero $p$-fields, the cosection-localized cycle is given by the Euler class of the obstruction sheaf coming from the $p$-field. Hence, by \cite[Proposition 6.2.2]{CKM} the two cycles coincide up to sign, where the sign change is due to taking the dual of tangent-obstruction complex.
\end{rema}
\begin{rema}\label{rema-0216}
We will define and show in Section  \ref{sec3} that the moduli stack $$\hqmapf^\delta:=\fQf^\delta\times_{\fQf}\hqmapf$$ of $\delta$-stable quasi-maps with $p$-fields has proper degeneracy locus $D(\sigma)$ if $d+\delta\ge 3(g-1)+m>0$. Therefore we obtain the GW-type invariant of $\delta$-stable quasi-maps to $\PP^4$ with $p$-fields by integrating against $ [\hqmapf^\delta]\virt_{\mathrm{loc}}.$
\end{rema}


\section{Moduli of stable pairs with multi-sections over a nodal curve}\label{sec2}
In this section, we review the notion of $\delta$-stability for a pair $(E,\alpha)$ of a coherent sheaf $E$ and a multi-section $\alpha:\cO_C^{\oplus n}\to E$ over a polarized nodal curve. We construct a projective moduli scheme of $\delta$-semistable pairs $(E,\alpha)$ over $C$ via geometric invariant theory. This type of construction is standard by now thanks to \cite{Huybrechts} and we will closely follow \cite{HL2}.

Let $C$ be a fixed $m$-pointed prestable curve of genus $g$. Let us fix an ample line bundle $\cO(1)$ on $C$. For a polynomial $P(t)\in \QQ[t]$ of degree $1$, we write either $P(t)=rt+\chi$ or $P(t)=r(t+\mu)$ in what follows. Let $\delta>0$ and $n\in \ZZ_{>0}$.

\begin{defi}
Let $E$ be a coherent sheaf on $C$ with Hilbert polynomial $P$ and $\alpha:\cO_C^{\oplus n}\to E$ be a homomorphism which we call a \emph{multi-section}.\begin{enumerate}
\item For a subsheaf $E'$ of $E$, let $\theta(E',\alpha)=1$ if $\alpha$ factors through $E'$ and $0$ if not. We write the Hilbert polynomial of $E'$ as $$P_{E'}(t)=\chi(E'(t))=r(E')t+\chi(E')=r(E')\left( t+\mu(E')\right)$$ where $E'(t)=E'\otimes \cO(1)^t$ and $\chi$ denotes the Euler characteristic. Let
\[ P_{E',\alpha}^\delta(t)=P_{E'}(t)+\theta(E',\alpha)\delta. \]
When $r(E')\ne 0$, we define the \emph{reduced Hilbert polynomial} of $(E',\alpha)$ as
\[ p_{E',\alpha}^\delta(t)=P_{E',\alpha}^\delta(t)/r(E')=\frac{P_{E'}(t)+\theta(E',\alpha)\delta}{r(E')}=t+\mu(E')+\theta(E',\alpha)\frac{\delta}{r(E')}. \]
\item For a quotient $q:E\to E''=E/E'$, let $\theta(E'',\alpha)=0$ if $q\circ\alpha= 0$ and $1$ if otherwise. We write $$P_{E''}(t)=\chi(E''(t))=r(E'')t+\chi(E'')=r(E'')\left( t+\mu(E'')\right),$$
\[ P_{E'',\alpha}^\delta(t)=P_{E''}(t)+\theta(E'',\alpha)\delta. \]
When $r(E'')\ne 0$, we define the \emph{reduced Hilbert polynomial} of $(E'',\alpha)$ as
\[ p_{E'',\alpha}^\delta(t)=P_{E'',\alpha}^\delta(t)/r(E'')=\frac{P_{E''}(t)+\theta(E'',\alpha)\delta}{r(E'')}=t+\mu(E'')+\theta(E'',\alpha)\frac{\delta}{r(E'')}. \]
\end{enumerate}
\end{defi}

\begin{defi}\begin{enumerate}
\item A $F$ on a scheme is called \emph{pure} if the support of any nonzero subsheaf of $F$ is of the same dimension as the support of $F$.
\item We say a subsheaf $E'$ of $E$ is \emph{saturated} if $E''=E/E'$ is a pure sheaf.
\end{enumerate}
\end{defi}

\begin{defi}\label{def-delstab1}
A pair of a one-dimensional sheaf $E$ on $C$ and a nonzero  homomorphism $\alpha:\cO_C^{\oplus n}\to E$ is \emph{$\delta$-semistable} if $E$ is pure and for any nontrivial subsheaf $E'\ne E$, \[ r(E) P^\delta_{E',\alpha} (t) \le r(E')P^\delta_{E,\alpha}(t) \]
i.e. $E$ is pure and $$\frac{\chi(E')+\theta(E',\alpha)\delta}{r(E')}\le \frac{\chi(E)+\delta}{r(E)}.$$
We get \emph{$\delta$-stability} if $\le$ is replaced by $<$.
\end{defi}
\begin{rema}
By definition, it is clear that $$P_{E',\alpha}^\delta+P_{E'',\alpha}^\delta=P_{E,\alpha}^\delta,\qquad r(E')+r(E'')=r(E).$$ Hence $\alpha:\cO_{C}^{\oplus n}\to E$ is $\delta$-semistable if and only if for any one-dimensional quotient $E''$ of $E$, \[ r(E) P^\delta_{E'',\alpha} (t) \ge r(E'')P^\delta_{E,\alpha}(t). \]
\end{rema}

\begin{defi}
A \emph{homomorphism} $\varphi:(E,\alpha)\to (E',\alpha')$ of pairs on $C$ is a homomorphism $\varphi:E\to E'$ of $\cO_C$-modules such that $\varphi\circ\alpha=\alpha'$.
An isomorphism $\varphi:E\to E'$ of $\cO_C$-modules which satisfies $\varphi\circ\alpha=\alpha'$ is called an \emph{isomorphism} of pairs.
\end{defi}

\begin{lemm} \label{lem-iso1}
Let $(E,\alpha)$ and $(E',\alpha')$ be two $\delta$-stable pairs of dimension $1$ on $C$ with the same reduced Hilbert polynomial $p^\delta_{E,\alpha}=p^\delta_{E',\alpha'}$. Then any nonzero homomorphism $\varphi:(E,\alpha)\to (E',\alpha')$ is an isomorphism.
\end{lemm}
\begin{proof}
Let $E_1$ denote the image of $\varphi:E\to E'$. We have $\varphi\circ \alpha=\alpha'$.
Let $q:E\to E_1$ be the surjective homomorphism induced from $\varphi$. Suppose $0\ne E_1\ne E'$. If $q\circ\alpha=0$, we have $\theta(E_1,\alpha)=0$ and
\[ p^\delta_{E,\alpha}(t)\le t+ \frac{\chi(E_1)}{r(E_1)}\le t+\frac{\chi(E_1)+\theta(E_1,\alpha')\delta}{r(E_1)}=p^\delta_{E_1,\alpha'}(t) <p^\delta_{E',\alpha'}(t) \]
which is a contradiction. If $q\circ\alpha\ne 0$, then
$\varphi\circ\alpha=\alpha'\ne 0$. Therefore $\theta(E_1,\alpha)=\theta(E_1,\alpha')=1$ and thus
\[ p^\delta_{E,\alpha}\le p^\delta_{E_1,\alpha}=p^\delta_{E_1,\alpha'}<p^\delta_{E',\alpha'}\]
which is also a contradiction. Hence for any nonzero $\varphi$, $\varphi$ is surjective. If $\mathrm{ker}\varphi\ne 0$,
\[
\frac{\chi(E)+\delta}{r(E)}<\frac{\chi(E')+\theta(E',\alpha)\delta}{r(E')}\le \frac{\chi(E')+\delta}{r(E')}
\]
and thus $p^\delta_{E,\alpha}<p^\delta_{E',\alpha'}$ which is a contradiction. Therefore, $\varphi$ is an isomorphism.
\end{proof}

\begin{rema}
(1) For a quotient $q:E\to E''$ of $E$, let $\alpha''=q\circ\alpha$. For a subsheaf $E'$ of $E$, let $\alpha':\cO_C^{\oplus n}\to E'$ be the homomorphism induced by $\alpha$ if $\alpha$ factors through $E'$ and let $\alpha'=0$ if not. We call $\alpha'=:\alpha_{E'}$, $\alpha''=:\alpha_{E''}$ the induced multi-sections of $E'$ and $E''$ respectively.

(2) Let $F\subset G\subset E$ be coherent sheaves and $\alpha:\cO_C^{\oplus n}\to E$ be a multi-section of $E$. Then the induced multi-section of $G/F$ as a quotient of $G$ is the same as the induced multi-section of $G/F$ as a submodule of $E/F$ because both are the compositions of $\alpha$ with the obvious projections.\end{rema}

\begin{prop}
Let $(E,\alpha)$ be a $\delta$-semistable one-dimensional pair on $C$ for some $\delta>0$. There exists a filtration \[ 0=E_0\subset E_1\subset E_2\subset \dots \subset E_s=E\]
called a \emph{Jordan-H\"older filtration} such that the factors $\mathrm{gr}_i(E)=E_i/E_{i-1}$ with the induced multi-section $\alpha_{E_i/E_{i-1}}$ are $\delta$-stable with the reduced Hilbert polynomial $p^\delta_{E,\alpha}$. Furthermore, the pair of \[ \mathrm{gr}(E)=\bigoplus_i \mathrm{gr}_i(E) \and \mathrm{gr}(\alpha)=\sum_i \alpha_{E_i/E_{i-1}} \]
is independent of the choice of the Jordan-H\"older filtration, up to isomorphism.
\end{prop}
\begin{proof} If $(E,\alpha)$ is $\delta$-stable, we are done. Suppose not.
Let $E'\ne E$ be a submodule of $E$ with $p^\delta_{E',\alpha}=p^\delta_{E,\alpha}$ and $r(E')<r(E)\in \ZZ$ maximal. Then $E/E'$ with the induced multi-section is $\delta$-stable. Thus by induction we obtain a Jordan-H\"older filtration. The uniqueness follows from Lemma \ref{lem-iso1}.
\end{proof}

\begin{defi}
Two $\delta$-semistable pairs $(E,\alpha)$ and $(E',\alpha')$ with $p^\delta_{E,\alpha}=p^\delta_{E',\alpha'}$ are called \emph{S-equivalent} if $(\mathrm{gr}(E),\mathrm{gr}(\alpha))\cong (\mathrm{gr}(E'),\mathrm{gr}(\alpha'))$.
\end{defi}

\def\bF{\mathbf{F} }

\begin{theo} Let $d>0$ and $C$ be a prestable curve with an ample line bundle $\cO(1)$. Let $P(t)=rt+\chi\in \ZZ[t]$. There is a projective scheme $\bF_{C}(P)^\delta$ which is a coarse moduli space for the functor which associates to a scheme $T$ the set of isomorphism classes of $T$-flat families of $\delta$-semistable pairs $(E,\alpha)$ on $C$ whose underlying sheaves have Hilbert polynomial $P$. Moreover, there is an open subscheme which represents the functor of families of $\delta$-stable pairs. A closed point in $\bF_C(P)^\delta$ represents an S-equivalence class of $\delta$-semistable pairs.\end{theo}

Since the construction is more or less standard, we only give a sketch of the proof. When $n=1$, the theorem is just a special case of \cite{LePotier}.
The first step of the proof is to show the boundedness of the collection of all $\delta$-semistable pairs with fixed Hilbert polynomial. By the $\delta$-semistability, the slope of a subsheaf
\[ \mu(E')\le \mu(E')+\theta(E',\delta)\frac{\delta}{r(E')}\le \frac{\chi+\delta}{r}\]
is bounded uniformly from above and the boundedness follows by \cite[Theorem 3.3.7]{Huybrechts}. From the boundedness, we obtain the $t$-regularity of the underlying sheaves of all $\delta$-semistable pairs for some uniform $t$ which enables us to identify $\delta$-semistable pairs with some orbits in the product $$W=\mathrm{Quot}_C(\CC^{P(t)}\otimes \cO_C(-t),P)\times \PP\Hom\left(H^0(\cO_C(t)^{\oplus n}),\CC^{P(t)}\right)$$ of the Quot scheme and the projective space.
By choosing an integer $l>0$, the Quot scheme is embedded into a projective space
\[ \mathrm{Quot}_C(\CC^{P(t)}\otimes \cO_C(-t),P)\hookrightarrow \PP \left( \wedge^{P(l)}\left(\CC^{P(t)}\otimes H^0(\cO_C(l-t))\right) \right)\]
and thus we have an ample line bundle $\cO_{\mathrm{Quot}}(1)$.
Then it is easy to see that the moduli space we desired is isomorphic to the GIT quotient of a closed subscheme $Z$ of $W$ by $PGL(P(t))$. The rest is the comparison of the GIT stability with the $\delta$-stability above. One should be careful about the choice of a linearized line bundle $\cO_{W}(n_1,n_2)$ on $W$ but it is okay to choose $(n_1,n_2)$ satisfying
\[ \frac{n_2}{n_1}=\frac{P(l)-P(t)}{P(t)+\delta}\cdot\delta \]
for $l$ sufficiently large (independent of $\delta$). We leave the detail to the reader because it is almost identical to the calculation in \cite{HL2}.

It is also standard to relativize the construction as in \cite[Theorem 4.3.7]{Huybrechts}.

\begin{theo}\label{thm-4.13}
Let $f:\cC\to S$ be a flat family of prestable curves and let $\cO_{\cC}(1)$ be a line bundle on $\cC$, ample relative to $S$. Let $P$ be a linear polynomial. Then there exists a scheme $\bF_{\cC/S}(P)^\delta$, projective over $S$ which universally corepresents the functor $$\cF^\delta: (Schemes/S)^\circ\lra (Sets)$$
which associates to an $S$-scheme $T$ the set of isomorphism classes of $T$-flat families of $\delta$-semistable pairs on the fibers of $\cC\times_ST\to T$ with Hilbert polynomial $P$. For closed points $s\in S$, the fiber  $\bF_{\cC/S}(P)^\delta|_s$ over $s$ is isomorphic to the moduli space $\bF_{C}(P)^\delta$ of $\delta$-semistable pairs on the fiber $\cC|_s$ over $s$.
\end{theo}


\section{Moduli of $\delta$-stable quasi-maps}\label{sec3}
In this section, we introduce the notion of $\delta$-stable quasi-maps on pointed prestable curves and prove that they form an open substack $\fQ^\delta_+$ of $\fQ_+$ which is a proper separated \DM stack for general $\delta>0$. By Theorem \ref{theo-sec1.5.1}, $\fQ_+^\delta$ has a perfect obstruction theory and thus a virtual fundamental class $[\fQ_+^\delta]\virt$. The GSW model $\fX_+^\delta=\fX_+\times_{\fQ_+}\fQ_+^\delta$ admits the localized virtual cycle $[\fX_+^\delta]\vilo$ whose support is contained in $\fQ_+^\delta$ and hence proper, for $d\ge 3(g-1)+m$.
Throughout this section, we let $2g-2+m>0$.

\begin{defi}\label{def:exc}
An $m$-pointed semistable curve $(C,q_1,\dots, q_m)$ is \emph{quasi-stable} if the length of any chain of rational bridges is at most 1 and $C$ has no rational tails. Here, a \emph{rational bridge} is a rational component with only two nodes and no marked points. A \emph{rational tail} is a rational component with only one node and one marked point. An \emph{exceptional component} means a rational bridge or a rational tail.
\end{defi}

\subsection{$\delta$-stable quasi-maps and their moduli}
Recall that a quasi-map to $\PP^4$ is a pair $(C,L,x)$ with $C$ a prestable curve, $L\in \mathrm{Pic}^d(C)$ and $x\in H^0(L^{\oplus 5})$.
\begin{defi}\label{def-delstab2}
For $\delta>0$, we say a quasi-map $(C,L,x)$ is \emph{$\delta$-stable (resp. $\delta$-semistable)} if the following hold:
\begin{enumerate}
\item $\olog_C\otimes L^\epsilon$ is ample for any $\epsilon>0$;
\item Given any line bundle $A$ on $C$ such that $\olog_C\otimes A^\epsilon$ is ample for any $\epsilon>0 $,  $x:\cO_C^{\oplus 5}\to L$ is $\delta$-stable (resp. $\delta$-semistable) with respect to $\olog_C\otimes A^\epsilon$ for $\epsilon>0$ arbitrarily small, in the sense of Definition \ref{def-delstab1}. The polarization $\olog_C\otimes A^\epsilon$ is used to define the Hilbert polynomial.
\end{enumerate}
\end{defi}

Recall that $\omega_{C}^\mathrm{log}:=\omega_{C}(\sum {q}_i)$ where $q_1,\dots, q_m$ are the smooth marked points on $C$. The first condition implies that $\olog_C$ is nef and thus $C$ is semistable admitting  a stabilization morphism $\rho:C\to \bar C$ to a stable curve which contracts rational components with only two special points, i.e. the exceptional components.

Since $\epsilon$ is arbitrarily small in the second condition of Definition \ref{def-delstab2}, this condition is enough to be checked for one line bundle $A$. The $\delta$-semistability implies that $C$ cannot contain a rational tail and that for any rational bridge $E\cong \PP^1$, $L|_E$ is $\cO_{E}(1)$. Indeed, by the first condition we must have $\deg L|_E>0$ for any exceptional component $E$. If $\deg L|_E\ge 1$ on a rational tail $E$ or $\deg L|_E\ge 2$ on a rational bridge $E$, then $L$ contains the destabilizing subsheaf $L'\cong \cO_{E}$. Note that $\chi(L')=1$ and $r(L')=\epsilon\cdot \mathrm{deg}(A|_E)$ so that $\mu(L')$ is arbitrarily large, and hence $L'$ is destabilizing regardless of $x$. By the same reason, no two exceptional components are adjacent. Therefore, the curve $C$ has to be quasi-stable (Definition \ref{def:exc}).

Let  $\rho:C\to\bar C$ be the stabilization morphism. In fact, by the following lemma we can replace the second condition of Definition \ref{def-delstab2} with
\begin{enumerate}
\item[(2\ensuremath{'})] $\rho_* x:\cO_{\bar C}^{\oplus 5}\to \rho_*L$ is a $\delta$-stable pair with respect to $\omega_{\bar C}^\mathrm{log}:=\omega_{\bar C}(\sum \bar{q}_i)$ for $\bar q_i=\rho(q_i)$.
\end{enumerate}

\begin{lemm}\label{lem-use}
Let $A$ be any line bundle on $C$ such that $\olog_C\otimes A^\epsilon$ is ample for any $\epsilon>0$. Suppose $(C, L,x)$ be a quasi-map such that $\olog_C\otimes L^\epsilon$ is nef for all $\epsilon>0$. Then $(L,x)$ is a $\delta$-stable pair on $C$ with respect to $\olog_C\otimes A^\epsilon$ for $\epsilon>0$ arbitrarily small if and only if the direct image $(\bar L,\bar x):=({\rho}_* L,{\rho}_* x)$ is a $\delta$-stable pair with respect to $\omega_{\bar C}^\mathrm{log}$.
\end{lemm}
\begin{proof}

Assume first that $(\bar L,\bar x)$ is a $\delta$-stable pair with respect to $\omega_{\bar C}^\mathrm{log}$. Since $\bar L$ is pure, the degree of $L$ restricted to any connected subcurve contracted by $\rho$ is either $0$ or $1$.
Suppose $L'$ is a destabilizing subsheaf of $L$ with respect to $\olog_C\otimes A^\epsilon$ for sufficiently small $\epsilon>0$. We may assume $L'$ is saturated. We claim that $\bar L':={\rho}_*L'$ is a destabilizing subsheaf of $\bar L$ with respect to $\omega_{\bar C}^\mathrm{log}$. Since the total degree of $L$ on any contracted chain is at most $1$, we have $\chi(\bar L)=\chi(L)$ and $\chi(L')=\chi(\bar L')$ by saturatedness. Since $\epsilon$ is arbitrarily small, $r(\bar L')$ (resp. $r(\bar L)$) is very close to $r(L')$ (resp. $r(L)$). Note that $r(\bar L')$ is a positive integer.  Since the inequality for the $\delta$-stability is strict, a very small perturbation does not change the inequality. Hence we find that $\bar L'=\rho_*L'$ is a destabilizing subsheaf of $\bar L$.

Conversely, assume that $(L,x)$ is $\delta$-stable with respect to $\olog_C\otimes A^\epsilon$ for any $\epsilon>0$. Then $0\le \deg L|_E\le 1$ for all exceptional component $E$ and hence $\bar L$ is pure.  Suppose $\bar L'$ is a destabilizing subsheaf of $\bar L$. Since $\rho^{-1}(q)$ is a chain of $\PP^1$'s with $\deg L|_{\rho^{-1}(q)}= 0$ or $1$ for any contracted node $q\in\bar C$, we can find a subsheaf $L'$ of $L$ such that $\rho_*L'=\bar L'$ such that $\theta(L',x)=\theta(\bar L',\bar x)$.
Again since strict inequality is preserved by small perturbation and $r(\bar L')$, $r(\bar L)$ are integers, we find that $L'$ is a destabilizing subsheaf of $L$.\end{proof}

\begin{defi}
Two quasi-maps $(C,L,x)$ and $(C',L',x')$ are called \emph{isomorphic} if there exist an isomorphism $\tau:C\to C'$ of $m$-pointed prestable curves and an isomorphism $\varphi:\tau^*L'\to L$ such that $\varphi\circ\tau^*x'=x$.
\end{defi}

\begin{defi}
A family of $\delta$-(semi)stable quasi-maps over $m$-pointed curves parameterized by a scheme $S$ consists of \begin{enumerate}
\item a flat family $\cC\to S$ of $m$-pointed quasi-stable curves and
\item an invertible sheaf $\cL$ on $\cC$ and a homomorphism $x_S:\cO_{\cC}^{\oplus 5}\to \cL$
\end{enumerate}
such that for every closed point $s\in S$, the fiber $(\cC_s,\cL_{s}, x_s)$ is a $\delta$-(semi)stable quasi-map to $\PP^4$.

Two families $x_S:\cO_{\cC}^{\oplus 5}\to \cL$  and $x'_S:\cO_{\cC'}^{\oplus 5}\to \cL'$ over $S$ are called isomorphic if there exist an isomorphism $\tau:\cC\to \cC'$ over $S$ and an isomorphism $\tilde{\varphi}:\tau^*\cL'\to \cL$ such that $\tilde{\varphi}\circ\tau^*x'_S=x_S$.
\end{defi}

\begin{defi}
Let $\fQ_+^\delta$ be the substack of $\fQ_+$ which associates to each scheme $S$ the groupoid
$\fQ_+^\delta(S)$
of families of $\delta$-semistable quasi-maps parameterized by $S$.
\end{defi}
\begin{rema}
Since ampleness and stability are both open conditions, the $\delta$-stable quasi-maps form an open subset for any family in $\fQ_+(S)$. Hence $\fQ_+^\delta$ is an open substack of $\fQ_+$.
\end{rema}
\begin{defi}
We say $\delta>0$ is \emph{general} with respect to a polynomial $P(t)=rt+\chi\in \ZZ[t]$ if there are no strictly $\delta$-semistable quasi-maps $(C,L,x)$ with $P_{\bar L}=P$ where $\rho:C\to \bar C$ is the stabilization morphism and $\bar L=\rho_*L$. Here strictly $\delta$-semistable means $\delta$-semistable but not $\delta$-stable. We say $\delta>0$ is a \emph{wall} if it is not general.
\end{defi}

In fact, there are only finitely many walls.
\begin{lemm}\label{lem-0219}
There are only a finite number of walls for $\fQ^\delta_+$.
\end{lemm}
\begin{proof}
We use Lemma \ref{lem-use}. Suppose $(L,x)$ is strictly $\delta$-semistable so that $(\bar L,\bar x)$ is a strictly $\delta$-semistable pair with respect to $\cO_{\bar C}(1):=\omega_{\bar C}^\mathrm{log}$. If a subsheaf $\bar L'$ of $\bar L$ has Hilbert polynomial $P_{\bar L'}(t)=r't+\chi'\in \ZZ[t]$ and $r'(P(t)+\delta)$ equals $rP_{\bar E'}(t)$ or $r(P_{\bar E'}(t)+\delta)$, then $\delta=\frac{r\chi'-r'\chi}{r'}$ or $\frac{r'\chi-r\chi'}{r-r'}.$ Note that $\bar L'$ is saturated in the sense that $\bar L /\bar L'$ is pure because otherwise the inverse image of the zero-dimensional part $T(\bar L/\bar L')$ in $\bar L$ has larger slope than $\bar L'$ and hence destabilizing.

In \cite[Section 6.1]{MOP}, it was proved that the collection of all $$(C, L)\in\mathfrak{P}_+$$ such that $\olog_C\otimes L^\epsilon$ is ample for any $\epsilon>0$ is bounded. Hence there are only a finite number of topological types (i.e. dual graphs of $C$ decorated by the degrees of $L$ restricted to irreducible components). For each  type, since $0\le r(\bar L')\le r(\bar L)$ and $\bar L'$ is saturated with $L$ invertible, there are only a finite number of possible pairs $(r(\bar L'),\chi(\bar L'))$. This certainly implies that there are at most finitely many walls.
\end{proof}

We are now ready to state the main result of this section. Note that $$P(t)=rt+\chi\in \ZZ[t]\and\chi=P(0)=\chi(\bar L)=\chi(L)=d-g+1$$ for $\delta$-semistable quasi-maps.

\begin{theo}\label{thm2-1} Suppose $2g-2+m>0$, $d+\delta\ge g-1$ and $\delta>0$ is general with respect to $P$. Then the open substack $\fQ_+^\delta$ is a proper separated \DM stack of finite type over $\CC$. Consequently, $\fQ_+^\delta$ has a perfect obstruction theory and a virtual fundamental class.
\end{theo}

\subsection{Proof of Theorem \ref{thm2-1}}\label{sec:pf}
In this subsection we prove Theorem \ref{thm2-1}.

We first prove that $\fQ^\delta_+$ is a \DM stack by using \cite[Corollary 2.2]{Edidin}: It suffices to write $\fQ_+^\delta$ as the quotient of a scheme by a reductive group and then show that the stabilizer groups are finite and reduced as a scheme. 

By \cite[Section 6.1]{MOP} and item (1) of Definition \ref{def-delstab2}, we find that if a quasi-map $(C,L,x)$ is $\delta$-semistable,
\[
(\olog_C)^{\otimes 5(d+1)}\otimes L^{ 5}
\]
is very ample. The complete linear system of $(\olog_C)^{\otimes 5(d+1)}\otimes L^{ 5}$ gives us an embedding $\imath:C\hookrightarrow \PP^N$ into a projective space and thus a point in the product $\mathrm{Hilb}\times (\PP^N)^m$ where $\mathrm{Hilb}$ denotes the Hilbert scheme of curves in $\PP^N$. Let $\cH$ be the locally closed locus in $\mathrm{Hilb}\times (\PP^N)^m$ of quasi-stable curves $(C,q_1,\dots, q_m)$. Let $\cC_\cH\to \cH$ denote the universal curve with sections $q_i$. Then
\begin{equation}\label{eq-02171}\omega_{\cC_\cH/\cH}^\mathrm{log}\otimes \cO_{\PP^N}(1)|_{\cC_\cH}^{\otimes \epsilon}\end{equation}
is relatively ample over $\cC_\cH$ for $\epsilon>0$ arbitrarily small. By Theorem \ref{thm-4.13}, there exists a scheme $\mathbf{F}_{\cC_\cH/\cH}(P)^\delta$ that parameterizes $\delta$-stable pairs $x:\cO_C^{\oplus 5}\to L$ with respect to \eqref{eq-02171}. Let $W$ be the locally closed subscheme of $\mathbf{F}_{\cC_\cH/\cH}(P)^\delta$ of $\delta$-stable pairs $x:\cO_C^{\oplus 5}\to L$ with $L$ invertible such that
\[ (\olog_C)^{\otimes 5(d+1)}\otimes L^{ 5}\cong \cO_{\PP^N}(1)|_C.\] Then $W$ parameterizes all $\delta$-stable quasi-maps. There is a natural action of $PGL(N+1)$ on $W$ and two points in $W$ represent isomorphic quasi-maps if and only if they lie in the same orbit. Therefore, we find that
\[
\fQ_+^\delta = W/PGL(N+1).
\]

\begin{lemm}\label{cor-5.2}
The stabilizer group of a $\delta$-stable quasi-map $(C,L,x)\in W$ is finite and reduced. Therefore $\fQ_+^\delta$ is a \DM stack.
\end{lemm}
\begin{proof}
Let $(L,x)$ be a $\delta$-stable pair with respect to \eqref{eq-02171}.
By definition, an automorphism of $L$ which preserves $x$ is the identity map since $x\ne 0$. This implies that the stabilizer group of $(C,L,x)$ is a subgroup of $\mathrm{Aut}(C)$. It is well known that $\mathrm{Aut}(C)$ is reduced as a scheme over $\CC$. Hence it is enough to show that the stabilizer group is finite. 
 
Infinite automorphisms of the underlying curve $C$ of a $\delta$-stable quasi-map can arise only from exceptional components, i.e. rational bridges. Recall that $C$ cannot contain a rational tail with one marked point by stability. Let $E$ be an exceptional component of $C$ and $L|_E\cong \cO_{\PP^1}(1)$. If the zero of $x|_E$ is away from the two nodes, a nontrivial automorphism of $E$ fixing the two nodes acts nontrivially on $x$ and hence no nontrivial stabilizer arises from $E$. Hence we may assume $x$ is zero at one or two of the nodes of some exceptional component $E$.

The group $\CC^*$ of automorphisms of $E$ fixing the two special points $0, \infty\in E$ acts on the fiber of $L|_{E}\cong \cO_{\PP^1}(1)$ over one special point $0\in E$ with weight $1$ and the other $\infty\in E$ with weight $-1$. We say $E$ is separating if $\overline{C-E}$ is disconnected. Then $\CC^*$ acts with weight $1$ (resp. $-1$) on $L$ restricted to the connected component of $\overline{C-E}$ containing $0$ (resp. $\infty$) if $E$ is separating. If $E$ is not separating, no automorphism of $E$ lifts to a stabilizer of $(C,L,x)$ because the automorphism action changes the line bundle $L$.

Let $E_1,\dots, E_l$ be separating exceptional components, at one or two of whose nodes $x$ is vanishing. Let $C^\dagger = \overline{C-\cup_{i=1}^l E_i}.$ If $x$ is zero on a connected component $C'$ of $C^\dagger$, then $(L,x)$ is unstable because either $(L',0)$ or $(L'',x)$ is destabilizing where $L'$ (resp. $L''$) is $L|_{C'}$ (resp. $L|_{{C^\dagger-C'}}$).

Suppose now that $x$ is nonzero on all connected components of $C^\dagger$.
We identify $\CC^*$ with the group of automorphisms of each exceptional component $E_j$ fixing two nodes. Consider the dual graph whose vertices are the connected components $C'$ of $C^\dagger$ and edges are the exceptional components $E_j$. For a vertex $C'$ of $C^\dagger$, let $s_j(C')=-1$ (resp. $+1$) if the automorphism group $\CC^*$ for $E_j$ acts with weight $-1$ (resp. $+1$). Then if we have a one-parameter subgroup of $(\CC^*)^l$ which fixes $(L,x)$, there exists $(a_j)\in \ZZ^l-\{0\}$ such that $\sum_{j=1}^l a_js_j(C')$ remains constant for any vertex $C'$. Let $C'$ and $C''$ be adjacent vertices of some edge $E_{j'}$. Then since $E_{j'}$ is separating, one can see that
 $$\sum_{j=1}^l a_js_j(C')=\sum_{j=1}^l a_js_j(C'')\pm 2a_{j'}.$$
 Since $(a_j)$ is a nonzero vector, there cannot be a one parameter subgroup of $(\CC^*)^l$ which fixes $(L,x)$.
Therefore, the stabilizer group is finite and reduced.
\end{proof}

Next we prove the separatedness. Let $0\in \Delta$ be a pointed smooth curve. Suppose there are two families of $\delta$-stable quasi-maps $(\cL_i,{x}_i)$ over $\cC_i\to \Delta$ for $i=1,2$  whose restrictions to $\Delta^*=\Delta-\{0\}$ are isomorphic. As explained in \cite[Section 6.2]{MOP}, possibly after base change ramified over $0$, we can find a family $\cC\to \Delta$ of pointed semistable curves and dominant morphisms $\pi_i:\cC\to \cC_i$ for $i=1,2$.
Since $\cC\to\Delta$ is projective, we can choose a relative ample line bundle $\cA$ over $\cC\to \Delta$ so that $\omega_{\cC/\Delta}^\mathrm{log}\otimes \cA^\epsilon$ is relatively ample for $\epsilon>0$. Then by Lemma \ref{lem-use}, both $(\pi_i^*\cL_i,\pi_i^*{x}_i)$ are families of $\delta$-stable pairs on $\cC\to S$ with respect to $\omega_{\cC/\Delta}^\mathrm{log}\otimes \cA^\epsilon$ for $\epsilon$ small enough. By Theorem \ref{thm-4.13}, we obtain two morphisms $\Delta\to \bF_{\cC/\Delta}(P)^\delta$ which coincide over $\Delta^*$. By the separatedness of $\bF_{\cC/\Delta}(P)^\delta$, we find that the two families $(\pi_i^*\cL_i,\pi_i^*{x}_i)$ should be isomorphic. However, by item (1) of Definition \ref{def-delstab2}, there cannot be a component $E$ in the central fiber of $\cC\to \Delta$ which is contracted by $\pi_i$ but not contracted by $\pi_j$ with $\{i,j\}=\{1,2\}$, because the degree of $\pi_i^*\cL_i|_E$ is $0$ and the degree of $\pi_j^*\cL_j|_E$ is positive by item (1) of Definition \ref{def-delstab2}. (Note that $E$ has only two special points.) This implies that the two families $\cC_i\to\Delta$ are isomorphic. By the separatedness of  $\bF_{\cC/\Delta}(P)^\delta$ again with $\cC=\cC_i$, we find that $(\cL_i,{x}_i)$ are isomorphic. This proves the separatedness.

\bigskip

Finally we prove the properness. Suppose we have a family of $\delta$-stable quasi-maps $(\cL^*,{x}^*)$ over $\cC^*\to \Delta^*$. We should extend it to $\Delta$.
After shrinking $\Delta$ if necessary, we may assume that the topological type of the fibers of $\cC^*\to \Delta^*$ is constant. As in \cite[Section 6]{MOP}, we normalize them and take the standard reduction of each connected component possibly after a base change ramified over $0$. In particular, we may assume that each component is a smooth surface. Upon gluing thus obtained families over $\Delta$ along the nodes, we obtain a family of semistable curves $\cC\to \Delta$ which extends $\cC^*\to\Delta^*$. Let $C_0$ denote the central fiber of $\cC\to \Delta$. Note that by construction, $\omega_{C_0}^\mathrm{log}$ is nef.
Since $\cC\to \Delta$ is projective, we can choose a relatively ample line bundle $\cA$. By Definition \ref{def-delstab2}, the family $(\cL^*,{x}^*)$ can be thought of as a family of $\delta$-stable pairs with respect to the ample line bundle $\omega_{\cC^*/\Delta^*}^\mathrm{log}\otimes \cA|_{\Delta^*}^\epsilon$ for $\epsilon$ sufficiently small. By the projectivity of $\bF_{\cC/\Delta}(P)^\delta$ over $\Delta$, we can extend this family to a family $(\cL,{x})$ of $\delta$-stable pairs parameterized by $\Delta$. Let $(L_0,x_0)$ be the $\delta$-stable pair over the central fiber $C_0$.

By construction, the normalization $\tilde{\cC}$ of $\cC$ is a disjoint union of smooth surfaces. Let $\rho:\tilde{\cC}\to\cC$ be the normalization map. Since $\cL$ is flat over $\Delta$ and pure on each fiber, $\cL$ is a torsion-free sheaf on $\cC$ so that the torsion-free part of the pullback $\tilde{\cL}$ of $\cL$ to $\tilde{\cC}$ is \[ \tilde{\cL}=\cF\otimes I_Z\] for some invertible sheaf $\cF$ and zero dimensional subscheme $Z$ of $\tilde{\cC}$. Since $\cL$ is pure on each fiber, we find by local calculation that $Z=\emptyset$ and hence $\tilde{\cL}$ is locally free. Let $q_\pm$ be sections of $\tilde{\cC}\to \Delta$ that are glued to the section $q:\Delta\to \cC$. Then $\cL$ along $q$ is recovered from the gluing homomorphism
\[
\tilde{\cL}|_{q_+}\oplus \tilde{\cL}|_{q_-}\mapright{(\psi_+,\psi_-)} \tilde{\cL}|_{q_+}\oplus \tilde{\cL}|_{q_-}/\cL|_q.
\]
Since $\tilde{\cL}$ is invertible, we may assume $\tilde{\cL}|_{q_\pm-q_\pm(0)}\cong \cO_{\Delta^*}\cong \cL|_{q-q(0)}$ after shrinking $\Delta$ if necessary. Let $(\psi_+^0,\psi_-^0):\tilde{\cL}|_{q_+-q_+(0)}\oplus \tilde{\cL}|_{q_--q_-(0)}\lra \cO_{\Delta^*}$ be the restriction of $(\psi_+,\psi_-)$ to $\Delta^*$. Since $\cL$ is locally free over $\Delta^*$, $\psi_\pm^0$ are surjective and we can extend $(\psi_+^0,\psi_-^0)$ to a homomorphism
\[
(\psi_+',\psi_-'): \tilde{\cL}|_{q_+}\oplus \tilde{\cL}|_{q_-}\lra \cO_{\Delta}(a0)
\]
for some integer $a$ such that $\psi_+'$ and $\psi_-'$ are not simultaneously vanishing at $0\in \Delta$. In case both $\psi_\pm'$ are not vanishing at $0$, the kernel $\hat{\cL}$ of the composite
\[ \rho_*\tilde{\cL}\lra \tilde{\cL}|_{q_+}\oplus \tilde{\cL}|_{q_-}\mapright{(\psi_+',\psi_-')} \cO_{\Delta}(a)
\] is locally free along $q$.  If $\psi_+'$ (resp. $\psi_-'$) is vanishing over $0\in \Delta$, we blow up $\tilde{\cC}$ at $q_-(0)$ (resp. $q_+(0)$) and let $\tilde{\cL}'=\pi^*\tilde{\cL}(-bE)$ where $\pi:\tilde{\cC}'\to \tilde{\cC}$ is the blow-up morphism with exceptional divisor $E$ and $b> 0$ is the vanishing order of $\psi_+'$ (resp. $\psi_-'$) at $q_+(0)$ (resp. $q_-(0)$).  By definition, $\tilde{\cL}'$ is locally free and we have
\[
(\psi_+'',\psi_-''): \tilde{\cL}'|_{q_+'}\oplus \tilde{\cL}'|_{q_-'}\lra \cO_{\Delta}((a+b)0)
\]
with $\psi_\pm''$ surjective where $q_\pm'$ are the proper transforms of $q_\pm$. By gluing $\tilde{\cC}'$ along $q'_\pm$, we obtain a family of nodal curves $\hat{\cC}\to \Delta$ and the kernel $\hat{\cL}'$ of the composite of $(\psi_+'',\psi_-'')$ with the restriction $\tilde{\cL}'\lra \tilde{\cL}'|_{q_+'}\oplus \tilde{\cL}'|_{q_-'}$ is an invertible sheaf on $\hat{\cC}$. Let $$\hat{\cL}=\hat{\cL}'\otimes \cO((b-1)E)$$
so that $\hat{\cL}|_E\cong \cO_{\PP^1}(1)$.
 It is easy to check that the multi-section $x$ of $\cL$ induces a natural multi-section $\hat{x}$ of $\hat{\cL}$ and the direct image of $\hat{\cL}|_0$ by $\hat{\cC}|_0\to \cC|_0$ is $\cL|_0=L_0$.

By Lemma \ref{lem-use}, item (2) of Definition \ref{def-delstab2} follows immediately from the above construction. So it only remains to prove item (1) of Definition \ref{def-delstab2}.

\begin{lemm}\label{lem0218}
Suppose $d+\delta\ge g-1$ so that $\chi+\delta\ge 0$. Let $x:\cO_C^{\oplus 5}\to L$ be a $\delta$-stable pair over a semistable curve $C$ with respect to some ample line bundle on $C$. Suppose that the rank of the pure sheaf $L$ on each component of $C$ is $1$. Then the degree of $L$ on each component of $C$ is nonnegative.
\end{lemm}
\begin{proof} Suppose $L$ has negative degree on a component $B$ of ${C}$. Let $L''$ be the torsion free part of $L|_B$. Then $\theta(L'',{x})=0$ because $L''$ admits no sections. Hence the reduced Hilbert polynomial of $L''$ is $t+\chi(L'')/r(L'')$. But since $h^0(L'')=0$, $\chi(L'')\le 0$. Therefore, if $\chi+\delta\ge 0$,
$$0\le\frac{\chi+\delta}{r}<\frac{\chi(L'')}{r(L'')}\le 0$$
by $\delta$-stability; a contradiction.
\end{proof}

By construction of $\hat{\cL}$, the nonnegativity of $L_0$ over components of $C_0$ is preserved while the degrees on the new exceptional components are all $1$. Therefore the restriction $\hat{L}_0$ of $\hat{\cL}$ to the central fiber $\hat{C}_0$ has nonnegative degrees on all components of $\hat{C}_0$.
If there is a component of the central fiber $\hat{C}_0$ where $\omega_{\hat{C}_0}^\mathrm{log}\otimes \hat{L}_0^\epsilon$ is trivial, we can simply contract those components by the line bundle $\omega_{\hat{\cC}/\Delta}^\mathrm{log}\otimes \hat{\cL}^\epsilon$ (tensored with the pull-back of a sufficiently ample line bundle of $\Delta$). Hence we proved item (1) of Definition \ref{def-delstab2}. This completes the proof of Theorem \ref{thm2-1}.

\subsection{Cosection localization for $\delta$-stable quasi-maps}
We can consider the GSW model for $\delta$-stable quasi-maps as follows.
Let $$\hqmapf^\delta=\fQf^\delta\times_{\fQf} \hqmapf$$ be the open substack of $\delta$-stable quasi-maps $(C,L,x)$ together with $p$-fields $p$. Since $\fQf^\delta$ is a \DM stack and the forgetful morphism $\hqmapf\to \fQf$ is representable by \cite[Proposition 2.2]{ChangLi}, $\hqmapf^\delta$ is a \DM stack as well.
\begin{theo}\label{thm:properdeg}
Suppose $g=0$ or $d+\delta\ge 3(g-1)+m>0$ and let $\delta>0$ be general. Then the degeneracy locus $D(\sigma)$ in $\hqmapf^\delta$ is proper and separated.
\end{theo}
\begin{proof}
Since $\fQf^\delta$ is proper and separated, it suffices to show that the degeneracy locus $D(\sigma)$ is contained in $\fQf^\delta$, that is, if $(C,L,x,p)\in D(\sigma)$, then $p=0$. When $g=0$, $H^0(C,L^{ -5}\otimes \omega_C)=0$ for any $(C,L,x,p)\in \hqmapf^\delta$ and hence we always have $D(\sigma)\subset \fQf^\delta$. So we suppose $g\ge1$ from now on.
By Lemma \ref{lem-0118}, the theorem follows if we show that the support of the image of $x$ contains the support of $p$ for any $(C,L,x,p)\in \hqmapf^\delta$.

Suppose $C'\subset C$ is an irreducible component in the closure of $\mathrm{supp}(p)-\mathrm{supp}(x)\ne \emptyset$.
Let $d'$ be the degree of $L|_{C'}$, $g'$ be the genus of $C'$, $m'$ be the number of marked points on $C'$ and $k'=\#(C'\cap\overline{C-C'})$.
Since $p:L^{ 5}\to \omega_C$ is nonzero on $C'$,
\[ 0\le 5d'\le \deg \omega_C|_{C'}=2g'-2+k'.\]
If $\deg \omega_C|_{C'}=0$, then $d'=0$ and  item (1) of Definition \ref{def-delstab2} fails. Hence $\deg \omega_C|_{C'}>0$ and in particular $C'$ is not contracted by the stabilization morphism $\rho:C\to \bar C$.
Then by Lemma \ref{lem-use} $\bar L$ is $\delta$-stable with respect to $\olog_{\bar C}$. Note that $r(L)=2g-2+m$, $\chi(L)=d-g+1$, $r(L|_{C'})=2g'-2+k'+m'$ and $\chi(L|_{C'})=d'-g'+1$. Since $x|_{C'}=0$, we have by $\delta$-stability
\begin{equation}\label{eq:p0}
\frac{d'-g'+1}{2g'-2+k'}\ge\frac{d'-g'+1}{2g'-2+k'+m'}>\frac{d+\delta-g+1}{2g-2+m}.
\end{equation}
If $g'\ge 1$, we have
\[ \frac15\ge \frac{d'}{2g'-2+k'}\ge \frac{d'-g'+1}{2g'-2+k'}>\frac{d+\delta-g+1}{2g-2+m}\]
and thus we find that $$d+\delta<\frac75(g-1)+\frac15 m\le 3(g-1)+m.$$
When $g'=0$, $k'\ge 3$ and we have
\[ \frac{d+\delta-g+1}{2g-2+m}<\frac{d'+1}{k'-2+m'}\le 1\] since $d'$ is an integer satisfying $0\le d'\le \frac15 (k'-2)$.
Hence $d+\delta<3(g-1)+m$. So in all cases we have $d+\delta<3(g-1)+m$ which contradicts our assumption.
\end{proof}
\begin{rema}\label{nre5}
When $\delta=0^+$ and $d\ge 3(g-1)+m$ and $(d-g+1,2g-2+m)=1$, then we have $\fX_+^{\delta=0^+}=\fQ_+^{\delta=0^+}$. Indeed, since there are no strictly $0^+$-semistable quasi-maps, the inequality \eqref{eq:p0} with $\delta=0$ must be satisfied for any component $C'$ regardless of $x$. Hence by the same proof, we have $p=0$.
\end{rema}



%

\section{$\delta$-stability on the LG side}\label{sec:LG}
In this section we describe a parallel theory on the LG side.

\subsection{FJRW theory by cosection localization}
We start by reviewing the definition of the FJRW invariant following \cite{CLL}. As mentioned in Section \ref{sec1}, we focus on the case of Fermat quintic $W=\sum_{i=1}^5 x_i^5$ and $G=\ZZ_5\subset \CC^*$. Let $\zeta:= e^{\frac{2\pi i}{5}}$ be the generator of $\ZZ_5$. We fix $g,m$ and $d$ throughout this section.

\begin{defi}
An $m$-pointed \emph{twisted curve} is a proper one-dimensional \DM stack with at worst nodal singularities together with $m$ distinct smooth marked points such that
\begin{enumerate}
  \item points with nontrivial stabilizers are marked points and nodes;
  \item all nodes are \emph{balanced}, i.e., locally near a node $\{zw=0\}$, the isotropy group $\ZZ_5$ acts by $(z,w)\mapsto (\zeta z, \zeta^{-1}w)$.
\end{enumerate}
\end{defi}

We assume the stabilizer at each marked point is $\ZZ_5$. It is well-known that  $m$-pointed twisted stable curves of genus $g$ form a proper separated \DM stack \cite{Chiodo}. For a twisted curve $C$, let $\tau:C\to |C|$ be the coarse moduli space. For a line bundle $L$ on $C$, we denote by $|L|$ the pushforward $\tau_*L$.

\begin{defi}
For a line bundle $L$ on a twisted curve $C$, the \emph{multiplicity} of $L$ at a marked point or a node is defined as follows. At a marked point, the local picture of $L$ is $(z,\lambda)\in \CC^2$ where $z$ is the coordinate along the curve $C$ and $\lambda$ is the coordinate along $L$. The multiplicity of $L$ at this marked point is defined to be the integer $\ell\in\{0,1,...,4\}$ such that the action of $\ZZ_5$ is given by $\zeta .(z,\lambda)= (\zeta z, \zeta^\ell\lambda)$. Similarly at a node, the local picture is $(z,w,\lambda)$, where $z$ and $w$ are coordinates along $C$ and $\lambda$ is the coordinate along $L$. Then the multiplicity of $L$ at the node is defined to be the integer $\ell\in\{0,1,...,4\}$ such that the action of $\ZZ_5$ is given by $\zeta.(z,w,\lambda)= (\zeta z, \zeta^{-1}w, \zeta^\ell\lambda)$. The multiplicity of $L$ at $q$ is denoted by $\mathrm{mult}_{q}L$.
\end{defi}

Let $\vec{k}=(k_1,\dots,k_m)$ be the $m$-tuple of integers with $0\le k_i<5$. We define the stack of $m$-pointed $\ZZ_5$-spin curves as
\[\overline{M}_{g,\vec{k}}=\{(C,L,p)~|~C\text{ twisted stable curve},~ p: L^{ 5}\overset{\cong}{\to} \otw_C,~ \mathrm{mult}_{q_i}L=k_i \}.\]
Here $\otw_C=\omega_C(\sum q_i)$ where $q_i$ are the orbifold marked points on $C$. In \cite{CLL}, $\otw_C$ was denoted by $\olog_C$ while in \cite{FJR15, RossRuan} by $\omega_{\mathrm{log}}$. We use different notation to emphasize that we allow simple poles only at \emph{orbifold marked points}. We denote by $\olog_C$ the sheaf of sections of $\omega_C$ possibly with simple poles only at \emph{smooth marked points}.
The notation $\otw_C$ is convenient when comparing the LG side with the CY side where we have to consider both orbifold and smooth marked points.

The stack $\overline{M}_{g,\vec{k}}$ is a smooth proper \DM stack and is nonempty if and only if $(2g-2+m-\sum k_i)$ is a multiple of 5 \cite{FJR}. By local computation, one can check that $|\otw_C|\simeq \olog_{|C|}$, where $\olog_{|C|}:=\omega_{|C|}(\sum \tau(q_i))$ and that for $(C,L,p)\in \overline{M}_{g,\vec{k}}$ the degree of $|L|$ is $d:=\frac15(2g-2+m-\sum k_i)$. (See \cite[Prop. 2.2.8]{FJR}.)
Note that the stability condition here implies the surjectivity of the homomorphism $p:L^5\to \otw_C$ which amounts to saying that $$|\epsilon|\cdot \mathrm{length}_z(\coker\, p) <1, \quad \epsilon<\!<0,\ \ \forall z\in C,$$
i.e. the triple $(C,L,p)$ is $\epsilon=-\infty$-stable in the sense of Section \ref{s0.3.1}.

By the same technique as in Section \ref{sec1.5}, Chang, Li and Li in \cite{CLL} constructed a GSW model for $\ZZ_5$-spin curves. Let $\fX_{-}^{\epsilon=-\infty}$ be the stack parametrizing $\epsilon=-\infty$-stable quadruples $(C,L,x,p)$, namely $(C,L,p)\in \overline{M}_{g,\vec{k}}$ and $x=(x_j)\in H^0(C,L)^{\oplus 5}$. Chang, Li and Li showed that ${\fX_{-}^{\epsilon=-\infty}}$ is a separated \DM stack of finite type.

When all $k_i$ are nonzero, $\vec{k}$ is called \emph{narrow}. Otherwise it is called \emph{broad}. When $\vec{k}$ is narrow, Chang, Li and Li constructed a cosection for ${\fX_-^{\epsilon=-\infty}}$ as follows:
The relative obstruction sheaf $Ob_{{\fX_-^{\epsilon=-\infty}}/\overline{M}_{g,\vec{k}}}$ at $(C,L,x,p)$ is given by $H^1(L)^{\oplus 5}$. When $\vec{k}$ is narrow, by \cite[Lemma 3.2]{CLL}, $H^1(L)^{\oplus 5}\cong H^1(L(-\sum q_i))^{\oplus 5}$. By using this identification, the cosection $\sigma$ is defined by
\beq\label{eq:lgcos}
(\dot{x_i})\mapsto p \sum 5x_i^4 \dot{x_i}.
\eeq
Since $(\dot{x_i})$ is regarded as an element in $H^1(L(-\sum q_i))^{\oplus 5}$ and $x_i\in H^0(L)$, the above map gives an element in $H^1(C,\omega_C)\cong \CC$. This cosection of the relative obstruction theory can be lifted to a cosection $\sigma:Ob_{\fX_-^{\epsilon=-\infty}} \to \cO_{\fX_-^{\epsilon=-\infty}}$. As before, we get a localized virtual cycle $[{\fX_-^{\epsilon=-\infty}}]\vilo$.

It is shown in \cite{CLL} that  $\sigma|_{(C,L,x,p)}=0$ if and only if $x_i=0$ for all $i=1,\dots, 5$, in other words, the degeneracy locus of $\sigma$ is $\overline{M}_{g,\vec{k}}$. In particular, the degeneracy locus is proper, and hence the invariant is defined by integrating against the localized virtual cycle $[{\fX_-^{\epsilon=-\infty}}]\vilo$. Chang, Li and Li \cite{CLL} showed that so defined invariants agree up to sign with the FJRW invariants defined by Fan, Jarvis and Ruan in \cite{FJR}.

\begin{rema}
The cosection \eqref{eq:lgcos} is parallel with the cosection on the CY side defined in \eqref{eq-hpot} where we considered the obstructions for $x$ and $p$ together. The obstruction for $p$ lies in $H^1(L^{-5}\otw_C)$ which is isomorphic to $H^1(\cO_C)$ since $p$ is an isomorphism. But this cancels with the deformation of the line bundle $L$. So, it is enough to consider only the obstruction for $x$. Note that the cosection \eqref{eq:lgcos} is the restriction of the cosection \eqref{eq-hpot}. However, for the other $\epsilon$- or $\delta$-stable quasi-maps we will discuss in next subsection, we will continue to use the cosection defined by the formula \eqref{eq-hpot}.
\end{rema}

\subsection{$\delta$-stable quasi-maps on the LG side}
Recall that we have the notion of $\epsilon<0$-stabilities on the LG side as defined in Section \ref{s0.3.1}.
In \cite{RossRuan}, Ross and Ruan studied $\epsilon$-wall crossing for the case $g=0$ on the LG side. They defined the \emph{twisted spin structure} by allowing the map $p:L^{ 5}{\to} \otw_C$ to be zero at finitely many smooth points. Ross and Ruan also derived the wall crossing formula.

In this section we introduce the $\delta$-wall crossing on the LG side. Let us denote $\tL:=L^{-5}\otimes\otw_C$.
Let $\fX_-$ (resp. $\fQ_-$) be the stack of quadruples $(C,L,x,p)$ (resp. triples $(C,L,p)$) of a twisted semistable curve $C$, a line bundle $L$ on $C$, a nonzero section $p\in H^0(C,\tL)$ and $x\in H^0(C,L^{\oplus 5})$.
We give an analogous definition for the $\delta$-stable quadruples $(C,L,x,p)$ (resp. triples $(C,L,p)$) as follows.

\begin{defi}
  For $\delta<0$, a quadruple $(C,L,x,p)$ (resp. triple $(C,L,p)$)  is \emph{$\delta$-(semi)stable} if the following hold:
  \begin{enumerate}
    \item $\otw_{|C|}\otimes |\tL|^a$ is ample for any $a>0$;
\item if $A$ is a line bundle on $C$ such that $\otw_{|C|}\otimes A^a$ is ample for all $a>0$, $(|\tilde{L}|, |p|)$ is $|\delta|$-(semi)stable with respect to $\otw_{|C|}\otimes A^a$ for $a>0$ arbitrarily small.
\end{enumerate}
Here $|\tilde{L}|=\tau_*\tilde{L}$ and $|p|=\tau_*p\in H^0(|C|,|\tilde{L}|)$ where $\tau:C\to |C|$ is the coarse moduli space. Note that $|\tilde{L}|$ is a line bundle on $|C|$ because $\tilde{L}$ is the pullback of a line bundle on $|C|$ as the orbifold structures on $\tilde{L}$ are trivial everywhere.
\end{defi}

As before we fix $\vec{k}=(k_1,\dots,k_m)$ and $d:=\frac15(2g-2+m-\sum k_i)$. For $\delta<0$, we let $\fX_-^\delta=\fX_-^\delta(g,\vec{k},d)$ (resp. $\fQ_-^\delta=\fQ_-^{\delta,\tw}(g,\vec{k},d)$) denote the stack of $\delta$-stable quadruples $(C,L,x,p)$ (resp. triples $(C,L,p)$) satisfying $\mathrm{mult}_{q_i}L=k_i$ and $\deg|L|=d$. The stack of pairs $(C,L)$ of a twisted curve $C$ and a line bundle $L$ on $C$ (without stability) is denoted by $\fP_-^{\tw}(g,\vec{k},d)$.

We now construct the moduli stack $\fX_-^\delta$ of $\delta$-stable quadruples as a separated \DM stack. The key point is that for any line bundle $L$ on a twisted stable curve $C$ with stabilizers $\ZZ_5$, $\tilde{L}=L^{-5}\otimes \omega_C^\tw$ has trivial orbifold structure and hence is the pullback of a line bundle on the coarse moduli space $|C|$ of $C$.

Let $\tilde{d}=-5d+2g-2+m$ and $\delta$ be general.
From Section \ref{sec3}, we have the moduli stack $\fQ^{\delta}_-(g,m,\tilde{d})$ of $\delta$-stable triples $(|C|,\tilde{L},p)$ where $\tilde{L}\in \mathrm{Pic}^{\tilde{d}}(|C|)$ and $p\in H^0(\tilde{L})$ with $p\ne 0$. Here the section $p$ plays the role of $x$ in Section \ref{sec3}. Although $x$ in Section \ref{sec3} is a 5-tuple of sections, the same theory works for one section $p$ (See Remark \ref{rmkn}).

The construction of $\fX^\delta_-$ can be summarized by the following diagram.
\begin{equation}\label{lgdiagram}
\xymatrix{
\fX^\delta_-=\bC(\pi_*\cL^{\oplus 5})\ar[d]\\
\fQ^{\delta,\tw}_{-}(g,\vec{k},d)\ar[r]^{\text{finite}} \ar[d] & \fQ^{\delta,\tw}_{-}(g,m,\tilde{d})\ar[r]^{\text{finite}} \ar[d] & \fQ^{\delta}_-(g,m,\tilde{d})\ar[dd] \\
\fP_-^{\tw}(g,\vec{k},d)\ar[dr]\ar[r]^{\text{finite}}_{L\mapsto L^{-5}\omega_C^\tw} & \fP_-^{\tw}(g,m,\tilde{d})\ar[d]\\
& \fM_{g,m}^{\tw}\ar[r]^{\text{finite}}_{C\mapsto |C|} &\fM_{g,m}^{\text{ps}}
}\end{equation}
Here $\bC(\pi_*\cL^{\oplus 5})$ denotes the direct image cone constructed in \cite{ChangLi}
and $\fX^\delta_-=C(\pi_*\cL^{\oplus 5})$ is the stack of quadruples $$\fX^\delta_-=\{(C,L,x,p)\,|\,(C,L,p)\in \fQ^{\delta,\tw}_-(g,m,d),  x\in H^0(L^{\oplus 5}), \mathrm{mult}_{q_i}L=k_i\}.$$

All the rectangles above are fiber products. The bottom right morphism sends each twisted prestable curve $C$ to its coarse moduli space $|C|$.
All vertical arrows are forgetful: The right vertical $(C,\tilde{L},p)\mapsto C$ is forgetting the line bundle $\tilde{L}=L^{-5}\omega_C^\tw$ and the $p$-field. The middle verticals $(C,\tilde{L},p)\mapsto (C,\tilde{L})\mapsto C$ forget the $p$-field and the line bundles successively. The left vertical $(C,L,p)\mapsto (C,L)$ forgets the $p$-field.

All the horizontal arrows are finite morphisms because the morphisms $C\mapsto |C|$ and $L\mapsto L^{-5}\omega_C^\tw$ are finite and the rest are base changes.

By Theorem \ref{thm2-1}, $\fQ^{\delta}_-(g,m,\tilde{d})$ is a proper separated \DM stack of finite type for general $\delta$. Therefore we find that $\fX^\delta_-$ is a separated \DM stack of finite type for general $\delta$.

To define a virtual cycle, we may apply the techniques of \cite{ChangLi} and cosection localization.
In \cite{Cheong, CCFK}, Cheong, Ciocan-Fontanine and Kim studied the deformation theory of line bundles on orbifolds and showed that the stack $\fP_-^\tw:=\fP_-^\tw (g,\vec{k},d)$ is smooth over $\mathfrak{M}_{g,m}^{\tw}$. Therefore we may apply \cite[Proposition 2.5]{ChangLi} to give a relative perfect obstruction theory
\[  \mathbb{L}^\vee_{\fX_-^\delta/\fP_-^\tw}\lra R\pi_*(\cL_{\fX_-^\delta}^{\oplus 5}\oplus [\cL^{ -5}_{\fX_-^\delta}\otimes \omega_{\pi}])      \]
over $\mathfrak{P}$ where $\pi:\cC_{\fX_-^\delta}\to \fX_-^\delta$ is the universal curve and $\cL_{\fX_-^\delta}$ is the universal line bundle.

We define the cosection by the same formula \eqref{eq-hpot}. Then we get a localized virtual cycle $[\fX_-^\delta]\vilo$.
\begin{theo}\label{thm:LGproper}
  Suppose $g=0$ or $-5d-\delta>g-1+m$ and let $\delta$ be general. Then the degeneracy locus in $\fX_-^\delta$ is proper and separated.
\end{theo}

\begin{proof}
The proof is analogous to that of Theorem \ref{thm:properdeg}. We show that the degeneracy locus is contained in $\fQ^{\delta,\tw}_-$ which is proper. When $g=0$, $H^0(L)=0$ for any $(C,L,x,p)\in \fX_-^\delta$. When $g\ge 1$, recall that in Theorem \ref{thm:properdeg}, we found a condition so that $\mathrm{supp}(p)-\mathrm{supp}(x)= \emptyset$. On the LG side, we need to have $\mathrm{supp}(x)-\mathrm{supp}(p)= \emptyset$.

Suppose $C'\subset C$ is an irreducible component in the closure of $\mathrm{supp}(x)-\mathrm{supp}(p)\ne \emptyset$. Let $d'$ be the degree of $L|_{C'}$, $g'$ be the genus of $C'$ and $k'=\#(C'\cap\overline{C-C'})$.

Since $x\in H^0(L)^{\oplus 5}$ is nonzero on $C'$, we have $d'\ge 0$.
Since $p|_{C'}=0$, by stability we have
\[
\frac{-5d'+g'-1}{2g'-2+k'}\ge\frac{-5d'+g'-1}{2g'-2+k'+m'}>\frac{-5d-\delta+g-1}{2g-2+m}.
\]

Then,
\[
\frac{-5d-\delta+g-1}{2g-2+m}<\frac{-5d'+g'-1}{2g'-2+k'}<\frac{g'-1}{2g'-2+k'}<1.
\]
So, we have $-5d-\delta<g-1+m$ which contradicts our assumption.
\end{proof}

\section{Comparison of moduli spaces when $g=0$}\label{sec:WCg0}
In this section, we study how the moduli spaces on the CY side when $g=0$ are related as the $\epsilon$- and $\delta$-stability conditions vary. We assume that  $\gcd(d+1,m-2)=1$ (See \eqref{eq:coprime}). When $g=0$, there are no nonzero $p$-fields and the cosection-localized virtual cycle is the Euler class of the obstruction sheaf coming from $p$-fields (Remark \ref{rem:cfk}). Therefore $\fX_+=\fQf$. Throughout this section, let $n$ be the size of the multi-section $x$, that is, $x\in H^0(L)^{\oplus n}$ (See Remark \ref{rmkn}).

\subsection{At $\delta=0^+$}
We denote by $\delta=0^+$ for $\delta>0$ sufficiently close to zero so that there is no other wall between $0$ and $\delta$. If $\gcd(d+1,m-2)=1$, then there are no strictly semistable line bundles $L$ on $C$. In such case by definition of $\delta$-stability, the quasi-map $(C,L,x)$ is $\delta=0^+$-stable if and only if $L$ is a stable line bundle.

\begin{prop}\label{prop:g0stable} Assume $g=0$ and $\gcd(d+1,m-2)=1$. Fix $(m,d,n)$ and an $m$-pointed quasi-stable curve $C$. Then there is a unique line bundle $L$ of degree $d$ on $C$ such that $(C,L,x)$ is a $\delta=0^+$-stable for some $x\in H^0(L)^{\oplus n}$, and for this $L$, $(C,L,x)$ is a $\delta=0^+$-stable for any nonzero $x\in H^0(L)^{\oplus n}$.
\end{prop}
\begin{proof}
  If $C$ is the irreducible $\PP^1$, then $L=\cO_{\PP^1}(d)$ and there is nothing to prove. When $g=0$, every node is a separating node. Fix a node $p\in C$. Let $C'$ and $C''$ be the two subcurves of $C$ such that $C'\cap C''=\{p\}$ and $C'\cup C''=C$. Let $d'$ (resp. $d''$) be the degree of $L|_{C'}$ (resp. $L|_{C''}$) and let $m'$ (resp. $m''$) be the number of marked points on $C'$ (resp. $C''$). Clearly, $d=d'+d''$ and $m=m'+m''$. Since $\olog_C\otimes L^\epsilon $ is ample for all $\epsilon>0$ by the $\delta$-stability, we have $m'\ge 2$ and $m''\ge 2$.

  It is enough to show that $d'$ and $d''$ is uniquely determined by the stability condition.

  By $\delta=0^+$-stability, we have
  \[ \frac{d'+1}{m'-1} > \frac{d+1}{m-2} \hspace{1em} \text{and} \hspace{1em} \frac{d''+1}{m''-1} > \frac{d+1}{m-2}.
  \]
  Then,
  \begin{equation}\label{eq:0stableg0}
  \frac{d+1}{m-2}(m'-1) -1 < d'< d- \frac{d+1}{m-2}(m''-1) +1 =\frac{d+1}{m-2}(m'-1).  \end{equation}
  Therefore $d'=\lfloor \frac{d+1}{m-2}(m'-1) \rfloor$ is the unique integer satisfying these inequalities. Similarly, $d''=\lfloor \frac{d+1}{m-2}(m''-1) \rfloor$. Therefore there is a unique line bundle $L$.
\end{proof}

\begin{rema}
The condition that there are no strictly semistable line bundles is essential in the proof of Proposition \ref{prop:g0stable}. For example, suppose that $C=C'\cup C''$ and $C'\cap C''=\{q\}$ and each of $C'$ and $C''$ is isomorphic to $\PP^1$ and has two marked points, Let $L_1$ and $L_2$ be the line bundles on $C$ such that $\deg(L_1|_{C'})=1$, $\deg(L_1|_{C''})=0$, $\deg(L_2|_{C'})=0$ and $\deg(L_2|_{C''})=1$. Then one can see that both of $L_1$ and $L_2$ satisfy \eqref{eq:0stableg0}.
\end{rema}

\begin{theo}\label{thm:g0delta0}
Assume $g=0$ and $\gcd(d+1,m-2)=1$. Then the moduli space $\hqmapf^{\delta=0^+}=\fQf^{\delta=0^+}$ is a projective bundle over $\overline{\cM}_{0,m}$.
\end{theo}
\begin{proof}
Let $(C,L,x)$ be a $\delta=0^+$-stable quasi-map. We first claim that $C$ is a stable curve. We know that $C$ is quasi-stable. Suppose that $C$ has a rational bridge $E$ with no marked points. Let $p_1$ and $p_2$ be nodes of $E$ and $C_1$ and $C_2$ be subcurves of $C$ such that $C=C_1\cup E \cup C_2$ and $C_i\cap E=\{p_i\}$ for $i=1,2$. By $\delta$-stability, the degree of $L$ on $E$ must be 1. Then by applying \eqref{eq:0stableg0} to two different decomposition $(C_1\cup E, C_2)$ and $(C_1, E\cup C_2)$, we get a contradiction. Hence $C$ does not have a rational bridge and hence is a stable curve in $\overline{\cM}_{0,m}$.

By Proposition \ref{prop:g0stable}, the line bundle $L$ is uniquely determined by $C$. Moreover since $L$ has a positive degree on each component, $h^1(C, L)=0$ and $h^0(C, L)=d+1$ by Riemann-Roch. Therefore, we see that $\hqmapf^{\delta=0^+}=\fQf^{\delta=0^+}$ is a projective bundle over $\overline{\cM}_{0,m}$ of rank $n(d+1)-1$.
\end{proof}

\subsection{Contraction morphisms}\label{sec:deltawc}

By Lemma \ref{lem-0219}, there are only finitely many walls for a fixed polynomial $P(t)=rt+\chi\in \QQ[t]$. Thus the $\delta$-line $(0,\infty)$ is partitioned into finite number of intervals on each of which $\fQ_+^\delta$ and $\fX_+^\delta$ stays constant. Let $\delta_0$ be a wall and $\delta_+>\delta_0>\delta_-$ be sufficiently close so that there are no other walls between $\delta_-$ and $\delta_+$.

As in the case of stable quotients \cite{MOP}, there is a natural morphism
\[ \mathfrak{q}_{\delta_0}: \fQ_+^{\delta_+}\lra \fQ_+^{\delta_-}.\]
Given $(C,L,x)\in \fQ_+^{\delta_+}$, the image $\mathfrak{q}_{\delta_0}(C,L,x)$ is obtained by the following construction. If $(C,L,x)$ is $\delta_-$-stable then trivially $\mathfrak{q}_{\delta_0}(C,L,x)=(C,L,x)$.

Let $(C,L,x)\in \fQ_+^{\delta_+}-\fQ_+^{\delta_-}$. Then there exists a subsheaf $\bar L'$ of $\bar L=\rho_*L$ (where $\rho: C\to \bar C$ is the stabilization) such that $\bar x$ does not factor
through $\bar L'$ and \beq\label{1201311} \frac{\chi'}{r'}=\frac{\chi+\delta_0}{r},\eeq
where $r't+\chi'\in \QQ[t]$ is the Hilbert polynomial of $\bar L'=\rho_*L'$ with respect to the ample line bundle $\olog_{\bar C}$ on $\bar C$. If $L'$ is not saturated, we replace it by a saturated subsheaf of $L$ containing $L'$ which is destabilizing. Hence $L'$ is completely determined by its support $C'$. We take maximal such a subcurve $C'$.

Let $C''$ be the complementary subcurve of ${C'}$ in $C$.
Then the quotient $L'':=L/L'$ is a sheaf on $C''$ and $\bar L''=\rho_*L''$ has Hilbert polynomial $r''t+\chi''$ with $r''=r-r'$ and $\chi''=\chi-\chi'$. Let $x''$ denote the image of $x$ by the quotient map $L\to L''$.
Since $(L, x)$ is $\delta_0$-semistable on $C$, $L'$ is a semistable sheaf on $C'$ and $(L'',x'')$ is $\delta_0$-semistable on $C''$.

If we fix $(C', L',0)$ and $(C'', L'', x'')$ as well as $C=C'\cup C''$, then the set of triples $(C,L,x)$ which are extensions of $(C'',L'',x'')$ by $(C',L',0)$ is $\PP H^0(C',{L'})^{\oplus n}$. All of these extensions belong to $ \fQ_+^{\delta_+}-\fQ_+^{\delta_-}$. To get a $\delta_-$-stable quasi-map, we take the following procedure, which is called the \emph{modification} of $(C,L,x)$ along $C'$. We exchange the subpair and the quotient pair and take a quasi-map in the extension $(C_1,L_1,x_1)$ of $(C',L',0)$ by $(C'',L'',x'')$. Such an extension is unique because in genus zero the line bundle $L$ is uniquely determined and the section $x$ is uniquely determined by $x''$. If $(C_1,L_1,x_1)$ has a rational bridge $E$ with no marked point such that $L_1|_E$ has degree 0, we contract $E$ and denote the resulting triple by $(C_1,L_1,x_1)$ again by slight abuse of notation. The quasi-map $(C_1,L_1,x_1)$ is called the modification of $(C,L,x)$ along $C'$.

We repeat this process until we get $\delta_-$-stable triple. If $(C_1,L_1,x_1)$ is not $\delta_-$-stable, take a maximal subcurve $C_1'$ which supports a saturated $L_1'$ satisfying \eqref{1201311} as before. Then $C_1'$ is a subcurve of $C'$ since otherwise $C'$ was not maximal. We modify along $C_1'$ to get $(C_2,L_2,x_2)$, which is again uniquely determined by $(C_1,L_1,x_1)$.
It is straightforward that after finitely many steps we get a $\delta_-$-stable triple, which by definition is the image $\mathfrak{q}_{\delta_0}(C,L,x)$. Therefore, the map $\mathfrak{q}_{\delta_0}: \fQ_+^{\delta_+}\lra \fQ_+^{\delta_-}$ is well-defined and it contracts the locus of $\delta_-$-unstable triples.

More precisely, the modification along $C'$ can be described as follows. Consider the moduli stack $\overline{\cM}_{0,m}$ of $m$-pointed stable curves of genus $0$. Let $\widetilde{\cM}_{0,m}$ be the blow-up of $\overline{\cM}_{0,m}$ so that the locus in $\overline{\cM}_{0,m}$ of curves which have subcurves $C'$ with Hilbert polynomials $r't+\chi'$ satisfying
\eqref{1201311} becomes a Cartier divisor. Let
\[ \widetilde{\fQ}_{+} = \fQ_+^{\delta_+}\times_{\overline{\cM}_{0,m}}\widetilde{\cM}_{0,m}.\]
Then the locus of $\delta_-$-unstable quasi-maps is a Cartier divisor $D$ on $\widetilde{\fQ}_{+}$.

Let $\widetilde{\cC}\to \widetilde{\fQ}_{+}$ and $\tilde x: \cO_{\widetilde{\cC}}^{\oplus 5}\to \widetilde{\cL}$ be the pull-back of the universal family of $\fQ^{\delta_+}_+$. Over the divisor $D$, the universal curve $\widetilde{\cC}|_D$ decomposes as $\cC'\cup \cC''$ such that for $\xi\in D$, the restriction $L''_\xi$ of $L_\xi=\widetilde{\cL}|_\xi$ to $C''_\xi=\cC''|_\xi$ gives the $\delta_-$-destabilizing quotient after stabilization. Let
\[\widetilde{x}_1:\cO_{\widetilde{\cC}}^{\oplus 5}\mapright{\widetilde{x}} \widetilde{\cL}\hookrightarrow \widetilde{\cL}(\cC')=:\widetilde{\cL}_1.\]
Then over $\xi\in D$, $\widetilde{L}_1$  is a line bundle which has $L''_\xi$ as a subsheaf and $\widetilde{x}_1|_\xi$ factors through $L''_\xi$. Now we contract all rational bridges $E$ with $\widetilde{L}_1|_E\simeq \cO_E$ to get a family  $(\widetilde{\cL}_1,\widetilde{x}_1)$ over $\widetilde{\fQ}_{+}$.
We repeat this process using $\widetilde{\fQ}_{+}$ in place of $\fQ_+^{\delta_+}$ until we get a family of $\delta_-$-stable quasi-maps.

Denote by $\widetilde{\fQ}_{+}^{\delta_0}$ the stack obtained by sequence of fiber products with the blow-ups of $\overline{\cM}_{g,m}$ on which we have a family $(\widetilde{\cL}_-,\widetilde{x}_-)$ of $\delta_-$-stable quasi-maps. Then $(\widetilde{\cL}_-,\widetilde{x}_-)$ gives us a morphism $$\widetilde{\mathfrak{q}}_{\delta_0}:\widetilde{\fQ}_{+}^{\delta_0}\lra \fQ_+^{\delta_-}.$$
Since the modification of a quasi-map is uniquely determined, this map factors through the projection $\widetilde{\fQ}_{+}^{\delta_0}\to \fQ_+^{\delta_+} $. Moreover, we have the commutative diagram
\beq\label{diadelpdelm}\xymatrix{
\fQ_+^{\delta_+} \ar[d] \ar[r]^{\mathfrak{q}_{\delta_0}} & \fQ_+^{\delta_-}\ar[d]\\
\overline{\cM}_{0,m} \ar@{=}[r]& \overline{\cM}_{0,m} .
}\eeq

\begin{exam}\label{ex:small}
When modifying along $C'$, we need to first blow up $\overline{\cM}_{0,m}$ along the locus of curves having $C'$ as a subcurve. This process is necessary because such locus can have codimension greater than one as the following example shows.

Let $C$ be the curve with 5 marked points having three rational components. Each component has three special points as shown in the picture below. The number labeled at each component is the degree of the line bundle $L$ restricted to that component.

\begin{center}
\begin{tikzpicture}[line cap=round,line join=round,>=triangle 45,x=0.2cm,y=0.2cm]
\clip(-2,-2) rectangle (38.5,12);
\draw [line width=1pt,domain=0:3] plot(\x,{(--10-3*\x)});
\draw [line width=1pt,domain=1:11] plot(\x,{(--2)});
\draw [line width=1pt,domain=9:12] plot(\x,{(-26 --3*\x)});
\begin{scriptsize}
\draw [fill=black] (1,7) circle (1.8pt);
\draw [fill=black] (1.67,5) circle (1.8pt);
\draw [fill=black] (6,2) circle (1.8pt);
\draw [fill=black] (11,7) circle (1.8pt);
\draw [fill=black] (10.33,5) circle (1.8pt);

\end{scriptsize}
\draw (6,-1) node {$\delta>2$};
\draw (-.5,2.2) node {$C'$};
\draw (0.3,5.5) node {$0$};
\draw (6,3.5) node {$3$};
\draw (12,5.5) node {$0$};

\draw [<->] (16,6.5)--(21,6.5) node[midway,below] {$\delta=2$};

\begin{scope}[shift={(25,0)}]\draw [line width=1pt,domain=0:3] plot(\x,{(--10-3*\x)});
\draw [line width=1pt,domain=1:11] plot(\x,{(--2)});
\draw [line width=1pt,domain=9:12] plot(\x,{(-26 --3*\x)});
\begin{scriptsize}
\draw [fill=black] (1,7) circle (1.8pt);
\draw [fill=black] (1.67,5) circle (1.8pt);
\draw [fill=black] (6,2) circle (1.8pt);
\draw [fill=black] (11,7) circle (1.8pt);
\draw [fill=black] (10.33,5) circle (1.8pt);

\end{scriptsize}
\draw (6,-1) node {$\delta<2$};
\draw (0.3,5.5) node {$1$};
\draw (12.5,2.2) node {$C'$};
\draw (6,3.5) node {$1$};
\draw (12,5.5) node {$1$};
\end{scope}
\end{tikzpicture}
\end{center}

Let $(C,L)$ be as in the left picture. For any section $x\in H^0(L)^{\oplus n}$, $(C,L,x)$ is a $\delta>2$-stable quasi-map. At $\delta=2$, this quasi-map needs to be modified along the middle component $C'$ since the saturated subsheaf supported on $C'$ has the Hilbert polynomial $t+2$ and has the maximal slope among all saturated subsheaves. After the modification, we will have the quasi-map on the right side where the modified section $\tilde x$ is zero away from $C'$.

Since the locus in $\overline{\cM}_{0,5}$ having such decomposition is of codimension two, this locus needs to be blown up to define the modification.

\end{exam}

\subsection{From $\epsilon=0^+$ to $\delta=\infty$}\label{sec:MOPinf}
In this subsection, we relate the moduli stack $\fX_+^{\epsilon=0^+}$ to $\fX_+^{\delta=\infty}$. Recall from \cite{CKM, MOP} (See also Section \ref{s0.3.1}) that  $\fX_+^{\epsilon=0^+}$ (resp. $\fQ_+^{\epsilon=0^+}$) is defined as the stack of quadruples $(C,L,x,p)$ (resp. triples $(C,L,x)$) satisfying
\begin{enumerate}
\item $\olog_C\otimes L^\epsilon$ is ample for all $\epsilon>0$;
\item the support of the cokernel of $x$ is 0-dimensional and disjoint from marked points and nodal points.
\end{enumerate}
By Definition \ref{def-delstab2} and Lemma \ref{lem-use},  $\fX_+^{\delta=\infty}$  (resp. $\fQ_+^{\delta=\infty}$) is the stack of quadruples $(C,L,x,p)$ (resp. triples $(C,L,x)$)  satisfying
\begin{enumerate}
\item $\olog_C\otimes L^\epsilon$ is ample for all $\epsilon>0$;
\item $(\bar L,\bar x)$ is an $\infty$-stable pair on $\bar C$ with respect to $\omega_{\bar C}$
where $\rho:C\to \bar C$ is the stabilization and $\bar L=\rho_* L$.
\end{enumerate}
Recall that an $\infty$-stable pair means a $\delta$-stable pair for sufficiently large $\delta$ such that there are no other walls larger than $\delta$. By Definition \ref{def-delstab1}, $(\bar L,\bar x)$ is an $\infty$-stable pair if and only if $\coker \bar x$ is zero dimensional. Hence the second condition (2) above can be rephrased as the following two conditions:
\begin{enumerate}
\item[($2'$)] the cokernel of $x:\cO^{\oplus 5}\to L$ has support in the union of rational bridges and finitely many points;
\item[($2''$)]  $C$ is quasi-stable and the restriction of $L$ to a rational bridge $E=\PP^1$ is $\cO(1)$.
\end{enumerate}
Therefore, if $(C,L,x)\in \fQ_+^{\epsilon=0^+}$, the only condition that may fail for $\delta=\infty$-stability is ($2''$) above. The locus in $\modmop$ having a rational tail $E$ (with one marked point) and with $\deg L|_E>0$ is obviously a divisor. The locus in $\modmop$ having a rational bridge $E$ with $\deg L|_E>1$ is a substack of codimension 2. Let $\widetilde{\fQ}_{+}^\infty$ be the blow-up of $\modmop$ so that the locus in $\modmop$ of $\delta=\infty$-unstable quasi-maps is a Cartier divisor $D$.

As in the previous subsection, let $(\widetilde{\cC},\widetilde{\cL},\widetilde{x})$ be the pull-back to $\widetilde{\fQ}_{+}^\infty$ of the universal family of $\modmop$. Then $\widetilde{\cC}|_D$ has a subcurve $\cE$ of exceptional components where the $\delta=\infty$-stability fails. Let
$$\widetilde{x}': \cO^{\oplus 5}_{\widetilde{\cC}}\lra \widetilde{\cL}\hookrightarrow \widetilde{\cL}(\cE)=:\widetilde{\cL}'.$$
Then one finds that the degrees of $L$ on the destabilizing rational bridges are decreased by 2 and the degrees of $L$ on the destabilizing rational tails are decreased by 1. Let $\cE'$ be the subcurve where the $\delta=\infty$-stability fails for the family $(\widetilde{\cC},\widetilde{\cL}',\widetilde{x}')$. Then we modify the family along $\cE'$. We continue this way until the degrees of $L$ restricted to exceptional bridges are either 0 or 1 and the degrees of $L$ restricted to exceptional tails are 0. Finally, we contract all exceptional components on which the degree of $L$ is 0. Therefore we obtain a family of $\delta=\infty$-stable quasi-maps parametrized by $\widetilde{\fQ}_{+}^\infty$ and thus a morphism $$\widetilde{\fQ}_{+}^\infty\lra \fQ^{\delta=\infty}.$$

As in the previous subsection, the modification of $(C,L, x)$ is uniquely determined. Hence the above map descends to the morphism
$$ \mathfrak{q}_\infty\colon\modmop\lra  \fQ_{+}^{\delta=\infty} $$
which fits into a commutative diagram
\beq\label{diastquoinfty} \xymatrix{
\modmop \ar[d] \ar[r]^-{\mathfrak{q}_\infty}  & \fQ_{+}^{\delta=\infty}\ar[d]\\
\overline{\cM}_{0,m}\ar@{=}[r] & \overline{\cM}_{0,m}}
\eeq

Combining the results of \cite{MOP}, Section \ref{sec:deltawc} and Section \ref{sec:MOPinf}, we have the following:
\begin{theo}\label{thm:g0}
When $g=0$, we have the contraction morphisms
\[
\fQ_{+}^{\epsilon=\infty} \lra \cdots \lra \modmop\lra  \fQ_{+}^{\delta=\infty} \lra \cdots \lra  \fQ_{+}^{\delta=0^+}.
\]
\end{theo}

\subsection{Comparison of virtual fundamental classes}
In this subsection, we compare the virtual cycles of the moduli spaces in Theorem \ref{thm:g0} prior to the cosection localization. In \cite{Mano}, Manolache proved that if $c:F\to G$ is a virtually smooth proper morphism of \DM stacks of same virtual dimensions such that $G$ is connected and for the relative perfect obstruction theory $E^\bullet$ we have $h^1/h^0(E^{\bullet\vee})\cong [E^1/E^0]$, then $$c_*[F]\virt=N[G]\virt$$ for some $N\in\QQ$. Using this, the following was shown

\begin{lemm}
[{\cite[Proposition 3.14]{Mano}}] \label{lem:mano}
If we have a commutative diagram
\[\xymatrix{ F\ar[r]^c\ar[d]_{\epsilon} & G\ar[d]^\nu \\
\fM_1\ar[r] &\fM_2}\]
such that \begin{enumerate}
\item $F$ and $G$ have the same virtual dimensions and $G$ is connected;
\item $\fM_1,\fM_2$ are smooth algebraic stacks;
\item $\epsilon,\nu$ have relative perfect obstruction theories $E_\epsilon, E_\nu$;
\item there is a morphism $c^*E_\nu\to E_\epsilon$ whose cone is a perfect complex of amplitude $[-1,0]$,\end{enumerate}
then $c_*[F]\virt=N[G]\virt$ for some $N\in \QQ$.
\end{lemm}

By \eqref{diadelpdelm} and \eqref{diastquoinfty}, we obtain the following.

\begin{prop}
 \[ \mathfrak{q}_{\delta_0*}[\fQ_{+}^{\delta_+}]\virt=[\fQ_{+}^{\delta_-}]\virt\and \mathfrak{q}_{\infty *}[{\fQ}_{+}^{\epsilon=0^+}]\virt=[\fQ_{+}^{\delta=\infty}]\virt.  \]
\end{prop}

\begin{proof}\label{prop:mano}
We prove the first equality. We use the notation in Section \ref{sec:deltawc}. Over the smooth algebraic stack $\mathfrak{P}_+$, we have the commutative diagram
\[\xymatrix{
\widetilde{\fQ}^{\delta_0}_+ \ar[d]_{p} \ar[dr]^-{\widetilde{\mathfrak{q}}_{\delta_0}} & \\
{\fQ}_{+}^{\delta_+} \ar[r]_-{\mathfrak{q}_{\delta_0}} & \fQ^{\delta_-}_+
}\]
Since $\widetilde{\fQ}^{\delta_0}_+$ is a fiber product, we have $p_*[\widetilde{\fQ}^{\delta_0}_+]\virt =[\fQ_{+}^{\delta_+}]\virt $. We apply Lemma \ref{lem:mano} to the diagram
\[\xymatrix{
\widetilde{\fQ}^{\delta_0}_+ \ar[d] \ar[r]^-{\widetilde{\mathfrak{q}}_{\delta_0}} & \fQ^{\delta_-}_+\ar[d]\\
\mathfrak{P}_+ \ar@{=}[r] & \mathfrak{P}_+
}\]

The modification in each step of the morphism $\widetilde{\mathfrak{q}}_{\delta_0}$ is given by
\[\widetilde{x}_1:\cO_{\widetilde{\cC}}^{\oplus 5}\mapright{\widetilde{x}} \widetilde{\cL}\hookrightarrow \widetilde{\cL}(\cC')=:\widetilde{\cL}_1.\]
They induces the morphism between the relative obstruction theories
\[ {\widetilde{\mathfrak{q}}_{\delta_0}}^* (R\pi_*(\widetilde{\cL}_1^{\oplus 5}))^\vee \to (R\pi_*(\widetilde{\cL}^{\oplus 5}))^\vee \]
which is the morphism in Condition (4) of Lemma \ref{lem:mano}.
The cone of this morphism is $R\pi_*(\widetilde{\cL}|_{\cC'}^{\oplus 5})^\vee $, which is perfect of amplitude $[-1,0]$.
Thus, by Lemma \ref{lem:mano}, $\widetilde{\mathfrak{q}}_{\delta_0*}[\widetilde{\fQ}_{+}]\virt=N[\fQ_{+}^{\delta_-}]$. Since the moduli spaces are isomorphic to each other on an open set, we have $N=1$.
and hence $\mathfrak{q}_{\delta_0*}[\fQ_{+}^{\delta_+}]\virt=[\fQ_{+}^{\delta_-}]\virt$. The proof of the second equality is similar.
\end{proof}

\begin{rema}
Let $\mathfrak{q}_{\epsilon}:\fQ_{+}^{\epsilon=\infty} \to \modmop$ be the contraction morphism constructed in \cite{MOP}. In \cite{MOP,Mano}, it is shown that $ \mathfrak{q}_{\epsilon *}[\fQ_{+}^{\epsilon=\infty}]\virt=[\modmop]\virt$. So, $\epsilon$- and $\delta$-wall crossing does not change the virtual cycles. However, Lemma \ref{lem:mano} and Proposition \ref{prop:mano} do not hold for cosection-localized virtual cycle.
\end{rema}


\subsection{Wall crossing for $d=1$}
In this subsection we study the example on the CY side when $g=0$ and $d=1$. We do not assume that $m$ is an odd integer. Throughout this section, let $\ell=\lfloor \frac{m-1}{2}\rfloor$.

Recall that the moduli space $\fQ_+^{\epsilon=\infty}$ is the moduli space $\overline{M}_{0,m}(\PP^{n-1},1)$ of stable maps to $\PP^{n-1}$ of degree 1. We study the $\epsilon$- and $\delta$-wall crossing.

For a $\delta$-stable quasi-map $(C,L,x)$, the degree of $L$ on each component of $C$ must be nonnegative. Hence the line bundle $L$ has degree one on only one component and has degree zero elsewhere. When $d=1$, destabilizing locus as in Example \ref{ex:small} does not appear. So the modification at each wall is a divisorial contraction.

\begin{lemm} \label{lem:delta1}
There are $\ell-1$ walls for the $\delta$-wall crossing.
\end{lemm}
\begin{proof}
A $\delta=\infty$-stable quasi-map $(C,L,x)$ of degree $1$ is $\delta=0^+$-unstable if there is a subcurve $C'$ of $C$ such that the degree of $L|_{C'}$ is 1 and $C'$ has $h$ marked points where $2\le h\le \ell$. In such a case, the saturated subsheaf supported on $C'$ is a destabilizing subsheaf and has Hilbert polynomial $(h-1)t+1$. Hence such quasi-maps needs to be modified at the wall $\delta_0=\frac{m-2}{h-1}-2$. Thus there are $\ell-1$ walls.
\end{proof}

\begin{prop}\label{prop:deltawc}
The contraction map $\mathfrak{q}_{\delta_0}:\fQ_+^{\delta_+}\to \fQ_+^{\delta_-}$ at each wall $\delta_0$ is a blowup.
\end{prop}
\begin{proof}
Let $(C,L,x)$ be as in the previous lemma. After the modification at a wall, we get a $\delta=0^+$-stable quasi-map  $(C,\tilde L,\tilde x)$ where $\tilde L$ has degree one on the complement $C''$ of $C'$ and $\tilde x$ is zero on $C'$. Note that since $\Ext^1((0,\cO_{C'}),(n, \cO_{C''}))=\CC$, where $(n, F)$ denotes the pair of the sheaf $F$ and a nonzero section in $H^0(F)^{\oplus n}$, there is a unique modification $(C,\tilde L,\tilde x)$. On the other hand, $\Ext^1((n, \cO_{C''}),(0,\cO_{C'}))=\CC^{n+1}$, which shows the fiber of the contraction map is isomorphic to $\PP^{n}$.

Let $\widetilde \fQ_+^{\delta_-}$ be the blowup of $\fQ_+^{\delta_-}$ along the locus $\Delta$ of $\delta_+$-unstable quasi-maps. Let $\widetilde \cC\to \widetilde \fQ_+^{\delta_-}$ and $\widetilde x : \cO_{\widetilde \cC}^{\oplus n } \to  \widetilde \cL$ be the pullback of the universal family on $\widetilde \fQ_ +^{\delta_-}$. Let $\cC''$ be the divisor on $\widetilde \cC$ such that $\widetilde \cL |_{\cC''}$ parametrize the destabilizing subpair. Let $\widetilde x' : \cO_{\widetilde \cC}^{\oplus n } \to  \widetilde \cL\to \widetilde \cL(\cC'')$ be the modification. Then $(\widetilde \cC, \widetilde \cL , \widetilde x')$ parametrizes $\delta_+$-stable quasi-maps and hence we have a morphism $\xi : \widetilde \fQ_+^{\delta_-} \to \fQ_+^{\delta_+}$. One can check that the normal bundle of $\Delta$ at the $\delta_+$-unstable quasi-map given by an element in $\Ext^1((0,\cO_{C'}),(n, \cO_{C''}))$ has the fiber  $\Ext^1((n, \cO_{C''}),(0,\cO_{C'}))$ (See the similar calculations in \cite{cc}). Therefore $\xi$ is an isomorphism and hence $\mathfrak{q}_{\delta_0}$ is a blowup morphism.
\end{proof}

\begin{theo} \label{thm:g0d1}
Suppose $g=0$ and $d=1$.
\begin{enumerate}
\item
\begin{enumerate}
\item When $m$ is odd, $\fQ_+^{\delta=0^+}$ is a $\PP^{2n-1}$-bundle over $\overline{\cM}_{0,m}$.
\item When $m=2\ell+2$ is even, $\fQ_+^{\delta=0^+}$ is a blowup of a $\PP^{2n-1}$-bundle over $\overline{\cM}_{0,m}$ along $\frac12\binom{m}{\ell+1}$ copies of a $\PP^{n-1}$-bundle over $\overline{\cM}_{0,\ell+2}\times \overline{\cM}_{0,\ell+2}$.
\end{enumerate} 
\item $\fQ_+^{\delta=\infty}$ is obtained from $\fQ_+^{\delta=0^+}$ by a sequence of $\ell-1$ blowups, where the blowup centers for the blowup at $\delta=\frac{m-2}{h-1}-2$ ($2\le h\le \ell$) is a disjoint union of $\binom{m}{h}$ copies of a $\PP^{n-1}$-bundle over $\overline{\cM}_{0,h+1}\times \overline{\cM}_{0,m-h+1}$
\item $\fQ_+^{\epsilon=0^+}$ is a blowup of $\fQ_+^{\delta=\infty}$ along $m$ copies of $\PP^{n-1}$-bundle over $\overline{\cM}_{0,m}$.
\item
\begin{enumerate}\item When $n\le 2$, $\overline{\cM}_{0,m}(\PP^{n-1},1)$ is isomorphic to $\fQ_+^{\epsilon=0^+}$.
 \item When $n >2$, $\overline{\cM}_{0,m}(\PP^{n-1},1)$ is a blow up of $\fQ_+^{\epsilon=0^+}$ along a $\PP^{n-1}$-bundle over $\overline{\cM}_{0,m+1}$.
\end{enumerate}
\end{enumerate}
\end{theo}
\begin{proof}
(1) We study $\fQ_+^{\delta=0^+}$. When $m$ is odd, $\fQ_+^{\delta=0^+}$ is a $\PP^{2n-1}$-bundle over $\overline{\cM}_{0,m}$ by Theorem \ref{thm:g0delta0}. We assume $m$ is even. Note that for a $\delta=0^+$-stable quasi-map $(C,L,x)$, the line bundle $L$ is not uniquely determined when $C$ is the union of two subcurves $C'$ and $C''$ meeting at one point $q$ each of which has $\ell+1$ marked points. Namely, $L$ can have degree 1 on any of two components containing $q$.

We consider the space $\fQ'$ of $(C,L,x)$ as follows. If $C$ is not the union of two subcurves having the same number of marked points, we choose $L$ so that $(C,L,x)$ with any nonzero multi-section $x$ is a $\delta=0^+$-stable quasi-map. By the proof of Theorem \ref{thm:g0delta0}, there is unique such $L$. If $C$ is a union of two subcurves meeting at a point $q$ having the same number of marked points, there are two components containing $q$. We choose one of two component continuously, for example, the component which is close to the first marked point $q_1$. We choose $L$ so that $L$ has degree 1 on the chosen component and $x$ is any nonzero multi-section. Then clearly $\fQ\rq{}$ is a $\PP^{2n-1}$-bundle over $\overline{\cM}_{0,m}$.

An element in $\fQ'$ fails to be $\delta=0^+$-stable if $C=C'\cup C''$ decomposes as above where $L$ has degree zero on $C''$ and $x$ is zero on $C''$. By modifying along such loci, we get the space $\fQ_+^{\delta=0^+}$. By the same argument as in Proposition \ref{prop:deltawc}, we see that $\fQ_+^{\delta=0^+}$ is a blowup of $\fQ\rq{}$.

(2) $\fQ_+^{\delta=\infty}$ is obtained from $\fQ_+^{\delta=0^+}$ by a sequence of $\ell-1$ blowups by Lemma \ref{lem:delta1} and Proposition \ref{prop:deltawc}. The blowup center for the blowup at $\delta=\frac{m-2}{h-1}-2$ ($2\le h\le \ell$) is a disjoint union of $\binom{m}{h}$ copies of a $\PP^{n-1}$-bundle over $\overline{\cM}_{0,h+1}\times \overline{\cM}_{0,m-h+1}$, where $\PP^{n-1}$ parametrizes the sections $\tilde x$ which vanish along $C'$ in the notation of Proposition \ref{prop:deltawc}.

(3) We now relate $\fQ_+^{\delta=\infty}$ with $\fQ_+^{\epsilon=0^+}$. A $\delta=\infty$-stable quasi-map $(C,L,x)$ is $\epsilon=0^+$-unstable if the base point lies on one of the marked points, say $q_i$. Recall that the base point is the point on which $x$ is zero. In this case, to get an $\epsilon=0^+$-stable quasi-map, we add a new rational component having the marked point $q_i$ which meets the rest of the curve $C$ at the position of $q_i$. Now the nonzero multi-section $x$ can be arbitrary. This procedure is also a blowup. So $\fQ_+^{\epsilon=0^+}$ is obtained from $\fQ_+^{\delta=\infty}$ by blowing up along $m$ copies of $\PP^{n-1}$-bundle over $\overline{\cM}_{0,m}$.

(4) Finally, we study the contraction map $\mathfrak{q}:\overline{M}_{0,m}(\PP^{n-1},1)=\fQ_+^{\epsilon=\infty}\to \fQ_+^{\epsilon=0^+}$. When $n\le 2$, it is not hard to show that the contraction map is an isomorphism (See \cite[Proposition 5.2]{CKcam}). So, we assume $n>2$. Let $(C,L,x)$ be a quasi-map in $\fQ_+^{\epsilon=0^+}$. Then the quasi-map $(\tilde C, \tilde L, \tilde x)$ that is in the fiber of $\mathfrak{q}$ is obtained as follows. Since the degree of $L$ is 1, there can be at most one base point. If $(C,L,x)$ does not have a base point, then $(\tilde C, \tilde L, \tilde x)=(C,L,x)$. Assume that $(C,L,x)$ has a base point $q$. Then $\tilde C$ is obtained from $C$ by adding a rational tail with no marked point at the base point and $\tilde L$ is the line bundle on this new curve $\tilde C$ having degree one on the added rational tail and having degree zero on $C$. The section $x\in H^0(L)^{\oplus n}$ having a base point comes from a section $x'$ in $H^0(L(-q))^{\oplus n}\simeq H^0(\cO_{C})^{\oplus n}$. The section $\tilde x \in H^0(\tilde L)^{\oplus n}$ is defined by extending $x'$ to $\tilde C$. On the moduli spaces, one can check that this process is also given by a blowup along the locus of quasi-maps having base point, which is isomorphic to $\PP^{n-1}$-bundle over $\overline{\cM}_{0,m+1}$ where the last marked point indicates the base point.
\end{proof}

Using Theorem \ref{thm:g0d1}, we can derive an explicit formula for the Poincar\'{e} polynomial of  $\overline{M}_{0,m}(\PP^{n-1},1)$ from the Poincar\'{e} polynomial of $\overline{\cM}_{0,m}$, where the latter is well known.

\begin{defi}
The  Poincar\'{e} polynomial of the space $X$ is
 $$P_t(X)=\sum_{i\ge 0}b_{2i}t^i,$$
 where $b_{2i}$ is the $2i$-th Betti number of $X$. Note that the odd Betti numbers of all the moduli spaces we consider here are zero.
\end{defi}

\begin{coro}\label{thm:d1poincare}
Let $P_m$ be the Poincar\'{e} polynomial of $\overline{\cM}_{0,m}$.
\begin{enumerate}
\item $ P_t(\fQ_+^{\delta=0^+})=\left\{\begin{array}{ll} \frac{1-t^{2n}}{1-t}P_m &\text{if $m$ is odd;}\\
\frac{1-t^{2n}}{1-t}P_m+\frac12\binom{m}{\ell+1} \frac{1-t^{n}}{1-t} P_{\ell+2}P_{\ell+2} \frac{t-t^{n+1}}{1-t}& \text{if $m$ is even.}\end{array}\right.$

\item The Poincar\'{e} polynomial of $\overline{M}_{0,m}(\PP^{n-1},1)$ when $n\ge 3$ is
\begin{align*}
P_t(\fQ_+^{\delta=0^+})+ \sum_{h=2}^\ell \left(\binom{m}{h} \frac{1-t^{n}}{1-t} P_{h+1}P_{m-h+1}\right)  \frac{t-t^{n+1}}{1-t}\\
+ m \frac{1-t^{n}}{1-t} P_m \frac{t-t^{n}}{1-t}+ \frac{1-t^{n}}{1-t} P_{m+1}\frac{t-t^{n-1}}{1-t}.
\end{align*}
When $n\le 2$, the same formula without the last term holds
\end{enumerate}
\end{coro}

\begin{rema} Since $\overline{\cM}_{0,3}$ is a point, we have $P_3=1$. For $m\ge 3$, there is a recursive construction of $\overline{\cM}_{0,m}$ in \cite{Keel}. Namely, $\overline{\cM}_{0,m+1}$ is obtained by blowing up $\overline{\cM}_{0,m}\times \overline{\cM}_{0,4}$ along codimension 2 loci each of which is isomorphic to $\overline{\cM}_{0,h+1}\times \overline{\cM}_{0,m-h+1}$ for some $2\le h\le \ell$. Therefore, we have the following recursive formula for $P_m$.

When $m=2\ell+1$ for $\ell\ge 1$,
\[
P_{m+1} = (1+t) P_m+\sum_{h=2}^\ell \left(\binom{m}{h} tP_{h+1}P_{m-h+1}\right),
\]
and when $m=2\ell+2$ for $\ell\ge 1$,
\[
P_{m+1} = (1+t) P_m+\sum_{h=2}^{\ell-1} \left(\binom{m}{h} tP_{h+1}P_{m-h+1}\right) + \frac12 \binom{m}{\ell}  tP_{\ell+1}P_{\ell+1}  .
\]

This formula can be rederived by using Corollary \ref{thm:d1poincare} as follows. For $n\ge 3$, we have a surjective morphism $\psi : \overline{M}_{0,m}(\PP^{n-1},1) \to Gr(2,n)$ sending a stable map to its image line in $\PP^{n-1}$. This is a fibration morphism where the fiber is isomorphic to $\overline{M}_{0,m}(\PP^{1},1) $. Hence we have
\[P_t(\overline{M}_{0,m}(\PP^{n-1},1) )= P_t(\overline{M}_{0,m}(\PP^{1},1) ) P_t(Gr(2,n)). \]
By rearranging this equation after substituting the result of Corollary \ref{thm:d1poincare}, we obtain the above recursive relation.
\end{rema}

\begin{exam}\label{exam:FM}
  When $n=2$, it is well-known that the moduli space $\overline{M}_{0,m}(\PP^1,1)$ is isomorphic to the configuration space $\PP^1[m]$ by Fulton and MacPherson \cite{FM}. Hence, the wall-crossing described above gives an alternative construction of the Fulton-MacPherson configuration space. Namely, $\PP^1[m]$ is given by a sequence of blow-ups of a $\PP^3$-bundle over $\overline{\cM}_{0,m}$. For example, the space $\PP^1[5]$ is obtained from $\PP^3$-bundle over $\overline{\cM}_{0,5}$ by a blowup along 10 disjoint $\PP^1$-bundles over $\overline{\cM}_{0,3}\times\overline{\cM}_{0,4}$ followed by another blowup along 5 disjoint $\PP^1$-bundles over $\overline{\cM}_{0,5}$. One can calculate the Poincar\'{e} polynomial of $\PP^1[5]$ as
  \[P_t(\PP^1[5])= 1+ 21t+ 67t^2+ 67t^3+ 21t^4+ t^{5},\]
  which agrees with the calculations in \cite{FM,KiemMoon,LiLi}.
\end{exam}

\section{A residue formula for $[\fX_\pm^{\delta=0^\pm}]\vilo$}\label{sec:residue}
Our main interest lies in the GW invariant and its comparison with the FJRW invariant. In this section, we compare the two invariants after $\epsilon$- and $\delta$-wall crossing. For this, we assume that there are no strictly $\delta=0^\pm$-semistable quasi-maps. Numerically, the condition is that
\begin{equation}\label{eq:coprime}
\gcd(d-g+1, 2g-2+m)=1\text{ and }\gcd(-5d+g-1+m, 2g-2+m)=1.
\end{equation}

By the dilaton equation \eqref{edila} and \cite[Theorem 4.2.9]{FJR}, for the GW and FJRW invariants, we can add marked points and cancel its effect by capping it with corresponding $\psi$ classes. Hence by adding more marked points if necessary, we assume the above numerical conditions throughout this section.
Consequently, the $\delta=0^\pm$-stability of $(C,L,x,p)$ is nothing but the stability of $\bar L$ as a sheaf on $\bar C$.

Let $\bar P=\bar{P}_{g,m,d}$ denote the moduli stack of pairs $(\bar C, \bar L)$, where $\bar C\in  \mgnst$ is a stable curve of genus $g$ with $m$ marked points and $\bar L$ is a (Gieseker-)stable sheaf with respect to the ample line bundle $\olog_{\bar C}$ which is of rank at most 1 on every component of $\bar C$ and of degree $d$. By the GIT construction of Simpson's in \cite{Simpson}, $\bar{P}$ is a proper \DM stack provided that there is no strictly semistable sheaf.  It is well known (see \cite{Alex} for instance) that the semistability condition for $\bar L$ is equivalent to the balanced condition for $L$ in \cite{Capo} when $m=0$.

We have the forgetful morphism $\bar P\to \mgnst$ where $\mgnst$ is the proper smooth stack of stable curves. Although the morphism $\bar P\to \mgnst$ is not smooth, $\bar P$ is also smooth.

\begin{lemm} \label{le6.3}
(1) $\bar P$ is a smooth \DM stack of dimension $4g-3+m$ which is proper when $\gcd(d-g+1,2g-2+m)=1$. \\
(2) Let $\fP^s$ denote the open substack of  the stack $\fP$ which consists of pairs $(C,L)$ with $C$ semistable and $L\in \mathrm{Pic}^d(C)$ such that $\olog_C\otimes L^\epsilon$ is ample for any $\epsilon>0$ and that  $(\bar C, \bar L)\in \bar P$ where $\rho:C\to \bar C$ is the stabilization and $\bar L=\rho_* L$. Let $\bar \fP^s=\fP^s/\CC^*$ where $\CC^*$ denotes the group of automorphisms of line bundles by scalar multiplications.
Then we have an isomorphism $\bar \fP^s\cong \bar P$.
\end{lemm}
\begin{proof} (1)  The proof is essentially due to Faltings in \cite[Theorem 4.1]{Faltings}.
Let $R$ be an Artin local ring over $\CC$ and $I$ be a square zero ideal of $R$. Let $\bar R=R/I$.
Suppose we have a stable curve $\bar \cC\to \Spec \bar R$ and a stable sheaf $\bar \cL$ on $\bar \cC$.
Then we can find a lift to $\Spec R$. Indeed, near a node in the central fiber, we may find $\bar p, \bar q\in \bar R$ and $\bar \pi=\bar p\bar q$ so that we can write $\bar \cC$ formally as $\bar R[\![x,y]\!]/(xy-\bar \pi)$ by \cite[Theorem 3.9]{Faltings}. Choose a lift $p, q\in R$ of $\bar p, \bar q$ and let $\pi=pq\in R$. Since the deformation of smooth variety is trivial, we can glue the trivial deformation of the smooth part with the deformation $R[\![x,y]\!]/(xy-\pi)$ of the node over $\Spec R$, to obtain a stable curve $\cC\to \Spec R$. By \cite[Theorem 3.9]{Faltings} again, the factorization $pq=\pi$ determines a torsion-free extension $\cL$ of $\bar \cL$ near the node and we can glue this with the trivial extension on the smooth part because the obstruction for gluing vanishes since the fiber dimension is 1.

(2) It is straightforward to check that $\bar \fP^s$ is \DM and the morphism $\bar \fP^s\to \bar P$ is birational and bijective. (For a set-theoretic inverse, we insert a rational bridge with $\cO(1)$ on it whenever we find a node where the torsion-free sheaf $\bar L$ on $\bar C$ is not locally free.)
The isomorphism follows from Zariski's main theorem because $\bar P$ is smooth by (1).
\end{proof}

\begin{prop}\label{p6.2n}
Suppose $d\ge 3(g-1)+m$ and $\gcd(d-g+1, 2g-2+m)=1$ so that $\bar P$ is a proper smooth \DM stack. Then
\beq\label{resfor}
[\fX_+^{\delta=0^+}]\vilo = \mathrm{res}_{t=0}\frac{[\bar P_{g,m,d}]}{e(R\pi_*(\cL^{\oplus 5}\oplus \cH om(\cL^{5},\omega_{\cC/\bar P})))}
\eeq
where $e(\cdot)$ stands for the Euler class of the perfect complex.
Here $\pi:\cC\to \bar P$ denotes the universal curve and $\cL$ is the universal sheaf on $\cC$.
\end{prop}
\begin{proof}
This is a consequence of the torus localization theorem for cosection-localized virtual cycles in \cite{CLK}.
By the proof of Theorem \ref{thm:properdeg} (See Remark \ref{nre5}), we find that if $(C,L,x,p)\in \fX_+^{\delta=0^+}$, then $p=0$. Hence $\fX_+^{\delta=0^+}=\fQ_+^{\delta=0^+}$ is the projectivization $\PP C(\pi_*\cL^{\oplus 5})$ of the cone $C(\pi_*\cL^{\oplus 5})$ in \eqref{econ}. Consider the compactification
$$C(\pi_*\cL^{\oplus 5})\cup \PP C(\pi_*\cL^{\oplus 5}) = \PP \left[C(\pi_*\cL^{\oplus 5})\oplus \cO\right]$$
of the cone with the obvious action of $\CC^*$ arising from the cone structure. The fixed loci are exactly the boundary at infinity and the zero section $\bar P$. Upon applying the torus localization formula in \cite{CLK} and taking the residue, we obtain  the proposition.
\end{proof}

Similarly on the LG side, when $\delta=0^-$, $\bar L$ must be stable since there are no strictly semistable sheaves on $\bar C$. Therefore, we have
$$\fX_-^{\delta=0^-}= \{(C,L,x,p)\,|\, \bar L \text{ is stable over } \bar C,  p\ne 0 \}.$$

\begin{prop}\label{p6.2LG}
Suppose $d<-\frac15 (g-1+m)$ and $\gcd(-5d+g-1+m,2g-2+m)=1$. Let $\tilde d= -5d+2g-g+m$.
Then,
\begin{equation}\label{eK2}
 [\fX_-^{\delta=0^-}]^{\rm vir}_{\rm loc} = r\cdot\mathrm{res}_{t=0}\frac{[\bar{P}_{g,m,\tilde d}]}{e(R\pi_*(\cL^{\oplus 5}\oplus  \cH om(\cL^{5},\omega_{\cC/\bar{P}})))},\end{equation}
where  $r$ is the degree of the finite morphism $\mathfrak{P}^{\rm tw}\to \mathfrak{P}^{\rm tw}$ sending $L$ to $\tilde L = L^{-5}\otw$
\end{prop}
\begin{proof}
When $(C,L,x,p)$ is $\delta=0^-$-stable and $-5d+\delta> g-1+m$, by the proof of Theorem \ref{thm:LGproper}, one can show that $x$ is always zero. So by the same argument as in Proposition \ref{p6.2n}, we obtain the same residue formula for $[\fQ^{{\delta=0^-}}]\vilo$.

After fixing the multiplicity vector $\vec{k}$, we see that $(\bar C,\bar L)$ uniquely determines $(C,\tilde L)$ by a local computation. So, by the diagram \eqref{lgdiagram}, we obtain \eqref{eK2}.
\end{proof}

\begin{rema}
The equations \eqref{resfor} and \eqref{eK2} are of the same form but have opposite ranges for $d$. So the wall crossing for  $[\fX_+^{\delta=0}]\vilo$ to  $[\fX_-^{\delta=0}]\vilo$ seems to require an analytic continuation as expected.
\end{rema}

\section{Insertions}\label{sec:insertion}
In this section, we define the invariants with insertions for $\delta$-stable quasi-maps. Note that since for $\delta$-stable quasi-maps the base points are allowed to lie on the marked points, there is no well-defined evaluation maps as in the $\epsilon$-stable quasi-map theory. Alternatively, we will define the invariants by directly imposing the conditions corresponding to the insertion on the moduli space.

For the insertion $\zeta_l=\prod_{i=1}^m \mathrm{ev}_i^*(h^{l_i})$ with $h=c_1(\cO_{\PP^4}(1))\in H^2(\PP^4)$,
consider the stack
\[
\fX(\zeta_l)=\{(C,L,x,p)\in \fX\,|\, x=(x_1,\dots,x_5), \ x_j\in H^0(C,L_j),\ p\in H^0(L^{-5}\omega_C)\}
\]
with $L_j=L(-\sum_{i=1}^m\lambda_{ij}q_i)$ where
\[
\lambda_{ij}=\left\{\begin{matrix} 0 & j>l_i\\ 1 & j\le l_i\end{matrix}\right.
\]
and $q_1,\dots, q_m$ are the marked points on $C$.
By the exact sequence
\beq\label{i32}
0\lra H^0(L_j)\lra H^0(L)\lra H^0(\oplus_{i=1}^m \CC_{q_i}^{\lambda_{ij}})\lra H^1(L_j)\lra H^1(L)\lra 0
\eeq
from the exact sequence $0\to L_j\to L\to \oplus_{i=1}^m \CC_{q_i}^{\lambda_{ij}}\to 0$, we find that $\fX(\zeta_l)$ is a closed substack of $\fX$.

Let $\pi^l:\cC(l)\to \fX(\zeta_l)$ be the universal curve and $\cL$ be the line bundle arising from the
morphism $\fX(\zeta_l)\hookrightarrow \fX\to \fP_{g,m,d}$. Let $\cL_j=\cL(-\sum_{i=1}^m\lambda_{ij}q_i)$.
By \cite{ChangLi}, the cone stack $\fX(\zeta_l)$ has
the relative perfect obstruction theory
\beq\label{i33}\def\bbL{\mathbb{L} }
\bbL_{\pi^l}^\vee\lra R\pi^l_*(\oplus_{j=1}^5\cL_j)\oplus R\pi^l_*(\cL^{-5}\omega_{\pi^l})
\eeq
which induces an absolute perfect obstruction theory of virtual dimension
\[
m-|l|=m-\sum_{i=1}^ml_i
\]
because $\fP_{g,m,d}$ is smooth.

By \eqref{i32} and \eqref{i33}, the relative obstruction space at a point $(C,L,x,p)$ of $\fX(\zeta_l)$ is
\[
H^1(\oplus_{j=1}^5L_j)\oplus H^1(L^{-5}\omega_C)
\]
which surjects onto the relative obstruction space at $(C,L,x,p)$ of $\fX$
$$H^1(L^{\oplus 5})\oplus H^1(L^{-5}\omega_C).$$
Thus it is straightforward that the cosection $\sigma$ for $\fX$ induces a cosection $\sigma_l$ of $\fX(\zeta_l)$ with
\[
\sigma^{-1}(0)=\sigma^{-1}_l(0) .
\]

Now we add stability. Let $\cU$ be an open separated \DM substack of $\fX$ such that $\cU\cap \sigma^{-1}(0)$ is proper.
For example, $\cU$ can be the substack of $\epsilon$ or $\delta$-stable quadruples $(C,L,x,p)$ in $\fX_+$ or $\fX_-$ considered in this paper. Let
\[
\cU(\zeta_l)=\cU\cap \fX(\zeta_l) .
\]
Then by the discussion above, we have a cosection-localized virtual fundamental class
\[
[\cU(\zeta_l)]\virt_{\mathrm{loc}}\in A_{m-|l|}(\cU(\zeta_l)\cap \sigma^{-1}(0))
\]
which has proper support.
We define the invariant with respect to the stability $\cU\subset \fX$ to be
\beq\label{i39}
\langle\!\langle \prod_{i=1}^m\tau_{a_i}(h^{l_i}) \rangle\!\rangle_\cU=\int_{[\cU(\zeta_l)]\virt_{\mathrm{loc}}} \prod_{i=1}^m\psi_i^{a_i}
\eeq

Note that when $\cU=\overline{M}_{g,m}(\PP^4,d)^p$ is the stack of stable maps with $p$-fields,
\[
[\overline{M}_{g,m}(\PP^4,d)^p]\vilo \cap \prod_{i=1}^m\psi_i^{a_i}\mathrm{ev}_i^*(h^{l_i})
=[\cU(\zeta_l)]\vilo\cap \prod_{i=1}^m\psi_i^{a_i}
\]
and hence \eqref{i39} coincides with the usual Gromov-Witten invariant with insertions.

In this way, we obtain invariants with insertions for open \DM substacks $\cU$ in $\fX$.

\bibliographystyle{amsplain}

\end{document}